\documentclass{article}
\usepackage[utf8]{inputenc}
\usepackage{amsfonts}
\usepackage{amsmath,amsthm}
\usepackage{amscd}
\usepackage[all,cmtip]{xy}
\usepackage{amssymb}
\usepackage{geometry}
\usepackage{fancyhdr}
\usepackage{lipsum}
\usepackage{mathtools}
\usepackage{relsize}
\usepackage{physics}
\usepackage{braket}
\usepackage{bm}
\usepackage{esint}
\usepackage{enumitem}
\usepackage{mathrsfs}
\usepackage{xcolor}
\usepackage{bbold}
\usepackage{comment}
\usepackage[
backend=biber,
style=alphabetic,
maxbibnames=99
]{biblatex} 
\allowdisplaybreaks

\usepackage{hyperref}
\hypersetup{
  colorlinks   = true,
  citecolor    = blue,
  linkcolor    = red
}

\numberwithin{equation}{section}

\newtheorem{theorem}{Theorem}

\newtheorem{corollary}[theorem]{Corollary}

\newtheorem{definition}[theorem]{Definition}

\newtheorem{lemma}[theorem]{Lemma}

\newtheorem{proposition}[theorem]{Proposition}
\newtheorem{remark}[theorem]{Remark}

\numberwithin{theorem}{section}

\def\text{\mbox}
\def\lra{\longrightarrow}

\def\bi{\begin{itemize}}
\def\ei{\end{itemize}}
\def\bem{\begin{emuerate}}
\def\eem{\end{enumerate}}
\def\beq{\begin{eqnarray*}}
\def\eeq{\end{eqnarray*}}

\def\R{\mathbb{R}} 
\def\P{\mathbb{P}} 
\def\E{\mathbb{E}} 
\def\S{\mathbb{S}} 

\def\1{\mathbb{1}}

\newcommand*\Laplace{\mathop{}\!\mathbin\bigtriangleup}

\newcommand{\lp}{\left(}
\newcommand{\rp}{\right)}
\newcommand{\lb}{\left[}
\newcommand{\rb}{\right]}
\DeclareMathOperator{\Hess}{Hess}
\newcommand{\bR}{\mathbb{R}}
\newcommand{\bS}{\mathbb{S}}
\newcommand{\eps}{\varepsilon}
\newcommand{\Vol}{\mathrm{Vol}}

\title{Sharp Riemannian heat kernel estimates on the cut locus and the Parabolic Anderson model}
\author{Hongyi Chen\thanks{Department of Mathematics, Aarhus University, Aarhus, Denmark. Email: \texttt{hchen77@math.au.dk}} \and Robert Neel\thanks{Department of Mathematics, Lehigh University, Bethlehem, Pennsylvania, USA. Email: \texttt{robert.neel@lehigh.edu}} \and Cheng Ouyang\thanks{Department of Mathematics, Statistics, and Computer Science, University of Illinois at Chicago, Chicago, Illinois, USA. Email: \texttt{couyang@uic.edu}}}

\date{}

\begin{document}

\maketitle

\begin{abstract}
\noindent{Using sharp global heat kernel bounds and geodesic comparison geometry, we show that the Dalang condition for well-posedness of the parabolic Anderson model with measure-valued initial conditions, first introduced on Euclidean space, holds on general compact Riemannian manifolds. We furthermore establish upper and lower moment bounds for all such solutions, providing evidence for intermittency in this generality. This extends and simplifies earlier work that required non-positive curvature.}
\end{abstract}

\tableofcontents

\section{Introduction}

Let \(M\) be a \(d\)-dimensional compact Riemannian manifold, and consider the parabolic Anderson model (PAM, sometimes called the stochastic heat equation with linear multiplicative noise) on \(M\), formally written as
\begin{equation}\label{eq: SHE}
    \begin{cases}
        \partial_t u(t,x)=\frac{1}{2}\Laplace_M u(t,x)+\beta u(t,x)\cdot \dot{W}, & t>0,\ x\in M,\\
        u(0,x)=\mu.
    \end{cases}
\end{equation}
Here, \(\beta>0\) is the (inverse) temperature parameter, \(\Laplace_M=\mathrm{div}(\mathrm{grad})\) denotes the Laplace--Beltrami operator on \(M\), \(\dot{W}=\dot{W}(t,x)\) is a family of space--time Gaussian noises to be specified later, and \(\mu\) is a finite measure on \(M\). A case of particular interest is when \(\mu=\delta_x\), the Dirac measure at a point \(x\in M\).

\medskip
The solution to \eqref{eq: SHE} has to be interpreted in the mild sense. That is, the solution to \eqref{eq: SHE} is understood as a (random field) solution to the following mild equation,
\begin{equation}\label{pam:mild}
    u(t,x)=\int_M P_t(x,y)\mu(dy)
    +\beta \int_0^t\int_M P_{t-s}(x,y)u(s,y)\,W(dy,ds)
    =J_\mu(t,x)+\beta I(t,x),
\end{equation}
where \(P_t(x,y)\) denotes the heat kernel on \(M\), and $J_\mu(t,x)$ refers to the solution to the homogeneous equation, namely,
\[
J_\mu(t,x):=\int_M P_t(x,y)\mu(dy).
\]
The stochastic integral \(I(t,x)\) is understood in the sense of It\^{o}--Walsh \cite{Walsh86,Dalang98}.

\medskip

The parabolic Anderson model arises in a wide range of problems in probability theory and mathematical physics. For instance, it is closely related to the free energy of directed polymers and to the Cole--Hopf solution of the KPZ equation \cite{ACQ11, Ka87, KPZ86}. It also has direct connections to the stochastic Burgers equation \cite{CM94} and to Majda’s model of shear-layer flow in turbulent diffusion \cite{Mj93}. Over the past few years, there has been significant progress in understanding fundamental properties of the PAM on flat spaces (e.g.\ \(M=\mathbb{R}^d\) and \(M=\mathbb{T}^d\)), including intermittency (see, e.g., \cite{Ch15, Kh14} and references therein), as well as energy landscapes and fluctuations \cite{BQS11, DGK23}. More recently, substantial advances have been made in the study of the critical stochastic heat flow, which can be viewed as a universal scaling limit of \eqref{eq: SHE} under appropriate renormalization in dimension 2 (see\cite{SHFbyCSZ,tsai2024SHF} and the references therein).

\medskip

The present paper approaches the PAM from a different perspective: we investigate how the geometry and topology of the state space influence the model and give rise to new phenomena. The underlying Brownian motion on the state space plays a crucial role in the behavior of the solution, and Brownian paths are themselves strongly affected by the geometry of the space. It is well known that Dalang’s condition is equivalent to the existence of a unique solution in the It\^{o} sense on Euclidean spaces. When \eqref{eq: SHE} is posed on Heisenberg groups, for bounded initial condition, it was shown in \cite{BOTW-Heisenberg} that Dalang’s condition naturally involves the Hausdorff dimension rather than the topological dimension of the space. Moreover, new Lyapunov exponents were identified in \cite{BCHOTW} for solutions to \eqref{eq: SHE} on metric measure spaces such as metric graphs and fractals, and the recent work \cite{Geng-Ouyang} further demonstrates that a novel phase transition occurs when \eqref{eq: SHE} is studied on hyperbolic spaces, again all for bounded initial conditions.

\medskip

Motivated by the close connection between the PAM and directed polymers, for which starting from the Dirac measure at a point is important, the present paper focuses on the PAM with measure-valued initial conditions. Indeed, this paper and the previous work of the first and third authors \cite{CO25} on compact manifolds with non-positive sectional curvature are the first to consider Dalang's condition for measure-valued initial conditions in geometric contexts. In contrast to the case of regular (function-valued) initial data, the analysis in this setting necessarily involves the global geometry of the underlying space, as we briefly explain below.

\medskip

Compact manifolds are also of intrinsic interest from the perspective of intermittency. Classical heuristic motivations for the parabolic Anderson model include localized concentration phenomena such as sunspots on the surface of the Sun, as well as models for galaxy formation and related clustering effects (cf. \cite{zeldovich90}). These examples naturally suggest studying the equation on spaces such as the sphere $\mathbb{S}^2$ rather than only on Euclidean space. While much of the earlier literature has focused on $\mathbb{R}^d$ as a mathematically simpler setting, the present work provides intermittency-related results directly on general compact manifolds. It is worth mentioning that \cite{CO25} addressed this question for compact manifolds with non-positive sectional curvature, but the present paper is the first to establish intermittency-related results for general compact manifolds, including spheres.

\medskip

Returning to the discussion of the global geometry involved in establishing Dalang's condition for measure-valued initial conditions, let \(P_{t,x,y}(s,z)\) denote the density of the Brownian bridge starting at \(x\) and reaching \(y\) at time \(t\). The analysis in
\cite{ChenThesis,ChenDalangAOP,CJKS11,Huang2016OnSH}, later further developed in \cite{COV23, CO25}, indicates that the well-posedness of \eqref{eq: SHE} hinges on a careful analysis of the following integral,
\begin{align}\label{intro: main quantity for iteratin}\int_0^t \int_{M^2}P_{t,x_0,x}(s,z)P_{t,x_0',x'}(s,z')\bm{G}_{\alpha,\rho}(z,z')m(dz)m(dz')ds.\end{align}
 Here $\bm{G}_{\alpha, \rho}$ is the spatial covariance function of the noise (see Section \ref{sec: noise} below for details). It is clear that the analysis of \eqref{intro: main quantity for iteratin} requires a precise understanding of how the law of a Brownian bridge concentrates for all \(t>0\), \(0<s<t\), and all \(x,y\in M\).
Since the Brownian bridge density admits the representation
\[
P_{t,x,y}(s,z)=\frac{P_s(x,z)P_{t-s}(z,y)}{P_t(x,y)},
\]
and one typically expects Gaussian-type heat kernel estimates (see Lemma~\ref{lem: heat kernel bd}  and Corollary \ref{Prop:LiYau} below), the concentration behavior is governed by the interaction of three distance functions arising from the exponential terms in the heat kernel bounds. This leads to the function
\begin{align}\label{func: F}
F_{s,t;x,y}(z)
:=-\frac{\bm{d}(x,y)^2}{2t}
+\frac{\bm{d}(x,z)^2}{2s}
+\frac{\bm{d}(z,y)^2}{2(t-s)}.
\end{align}

This is where global geometry enters and creates the main analytic difficulty. The function \(F\) attains its minimum when \(z\) lies on a distance-minimizing geodesic connecting \(x\) and \(y\), and for small \(t\) the Brownian bridge measure is therefore highly concentrated near these minimizers. When \(x\) and \(y\) belong to each other’s cut locus, there may be multiple—possibly infinitely many—distance-minimizing geodesics connecting them, some of which may be conjugate, which makes the analysis of \(F\) substantially more delicate.

\medskip

To address this difficulty, \cite{CO25} imposed the assumption that the sectional curvature of \(M\) is non-positive. This curvature condition plays two key roles:
\begin{itemize}
\item[(1)] it guarantees that there are only finitely many distance-minimizing geodesics between any two points \(x\) and \(y\), simplifying the analysis;

\item[(2)] it places the universal cover of \(M\) in the class of \(CAT(0)\) spaces, allowing one to compare geodesic triangles in \(M\) with Euclidean triangles. As a consequence, one can show that
\begin{align}\label{eq: hebavior of $F$ in sausages}
F_{s,t;x,y}(z)\geq \max_i \bm{d}(z,x_i^*)^2,
\end{align}
where \(x_i^*\) are the minimizers of \(F\) along the distance-minimizing geodesics connecting \(x\) and \(y\).
\end{itemize}

The main contribution of the present paper is to establish well-posedness of \eqref{eq: SHE} on \emph{general compact Riemannian manifolds}, without imposing any sign condition on the sectional curvature. In particular, we allow the presence of conjugate points and thus the possibility that there exist infinitely many distance-minimizing geodesics between two points, a regime that is excluded by the non-positive curvature assumption in \cite{CO25}. This requires new ideas to control the concentration behavior of Brownian bridges beyond the $CAT(0)$ framework and to analyze the function $F$ in \eqref{func: F} in the presence of highly non-unique geodesics.

\medskip
More precisely, we show that assuming Dalang's condition, and for arbitrary finite measure-valued initial data, the parabolic Anderson model admits a random field solution with sharp moment bounds. The following theorem summarizes our main well-posedness and regularity results. 

\begin{theorem}\label{thm: main result}
    Let $M$ be any compact manifold and $W=W_{\alpha,\rho}$ be the colored noise of Definition~\ref{def: G_alpha etc}. For any $\alpha>{(d-2)}/{2}$ and any finite measure $\mu$ on $M$, equation~\eqref{eq: SHE} admits a unique random field solution $\{u(t,x)\}_{t>0,x\in M}$ in the sense of Definition~\ref{def: SHE mild}. The solution is $L^p(\Omega)$-continuous for all $p\geq2$ and
     there exist constants $C, \theta>0$, depending on $\alpha,\beta$ and $p$, such that
    \[
    \E[|u(t,x)|^p]^{\frac{1}{p}}
    \leq CJ_\mu(t,x)e^{\theta t}.
    \]
    Here $J_\mu(t,x):=\int_M P_t(x,y) \mu(dy)$ is the homogeneous solution to the heat equation starting from $\mu$. 
\end{theorem}

In contrast to suggestions in earlier literature, note that here we have the same Dalang condition, namely $\alpha>(d-2)/2$, as one has for bounded initial data in the Euclidean case. Thus, if one is interested in  geometries where the Dalang condition for measure-valued initial conditions is different than for bounded initial conditions, one likely needs to look past Riemannian manifolds, perhaps to sub-Riemannian manifolds or fractals.

\medskip
 Going beyond previous work, this theorem requires that we treat the geometry of Riemannian geodesics and their cut loci in full generality, allowing for the possibility that there exist infinitely many distance-minimizing geodesics connecting \(x\) and \(y\), as well as conjugate geodesics. New ideas are therefore required to analyze the function \(F\) in \eqref{func: F} and to understand the concentration behavior of Brownian bridges on \(M\). The main ingredients and novelties of our approach are as follows:
\begin{itemize}
\item[(1)] When \(x\) and \(y\) are sufficiently close (so that the connecting geodesic is unique), we show that there exists a constant \(\zeta\in(0,1)\), depending only on \(M\), such that
\begin{equation}\label{Intro: Eqn:LemmaZeta}
F_{s,t;x,y}(z) \geq \zeta\, \bm{d}(x^*,z)^2,
\end{equation}
for all \(z\in M\), where \(x^*\) denotes the unique minimizer of \(F\) along the geodesic connecting \(x\) and \(y\) (See Lemma \ref{lem: F concentration no cutlocus} for a more precise statement). This result makes essential use of the Riemannian structure, relying on triangle comparison governed by an upper bound on the sectional curvature.

\item[(2)] When \(x\) and \(y\) are far apart, and thus possibly in the cut locus (in which case there may be multiple, even infinitely many, and/or conjugate distance-minimizing geodesics ), we show that for all \(s\in(0,t)\) and all \(x,y,z\in M\), one has (see Lemma \ref{lem: F concentration} below)
\[
F_{s,t;x,y}(z)\geq \left(\bm{d}(x,z)-\frac{s}{t}\bm{d}(x,y)\right)^2.
\]

The above lower bound is globally sharp when \(M\) is a sphere. In particular, when \(x\) and \(y\) are antipodal points, the Brownian bridge measure concentrates along lines of latitude, forming a submanifold of positive dimension. In this situation, Li--Yau type heat kernel bounds are insufficient, and a sharper upper bound on the heat kernel (see Lemma~\ref{lem: heat kernel bd}) is required for the analysis.
\end{itemize}

{
We also note that once the proper analysis of $F$ is done and an appropriate variant of the quantity \eqref{intro: main quantity for iteratin} has been bounded, it becomes obvious that the standard induction argument used to obtain the ingredients for the moment Picard iteration using only a single heat kernel bound \cite{CH19,ChenKim19,CCL23,COV23,CO25} does not work in this setting. 
However, thanks to the relationship between the heat kernel bounds of Proposition \ref{Prop:LiYau} and Lemma \ref{lem: heat kernel bd}, we are able to take advantage of a previously unexploited flexibility in the structure of the induction argument and modify the induction effectively (see Section \ref{Sec Ln bd}).
}

\begin{remark}
The analysis of \eqref{intro: main quantity for iteratin} is also closely related to estimates of the Malliavin derivative of the solution $u$, which has been an important ingredient in proving certain central limit theorems for $u$ (see, e.g., \cite{CKNP21}). The analysis developed in the present paper provides such estimates on non-Euclidean spaces and opens the door to studying corresponding central limit theorems in these settings, a direction that will be pursued in future work.
\end{remark}

\medskip

The second main result of the paper establishes a matching exponential lower bound on moments of the solution; such moment bounds are sometimes referred to as ``weak intermittency'' (see, for example, \cite{SamyWeakInt}).
\begin{theorem}\label{intro: thm: lower bound}
   Assume the same conditions as in Theorem \ref{thm: main result}. In addition assume $\rho>0$ for the noise $W_{\alpha,\rho}$ introduced in Definition \ref{def: G_alpha etc}.  Then for any $p\geq2$,
    \[
    \liminf_{t\uparrow\infty} \frac{1}{t}\ln\E[u(t,x)^p]>0.
    \]
\end{theorem}

This result generalizes the lower bound obtained in \cite{CO25} from function-valued initial conditions to measure-valued initial conditions. The proof builds on the arguments developed in \cite{CO25}, together with recent progress on (strict) positivity of solutions to \eqref{eq: SHE} on non-Euclidean spaces \cite{FSY26}. The argument shows that the exponential growth rate is ultimately driven by the ergodicity of Riemannian Brownian motion on compact manifolds.

The rest of the paper is organized as follows. In Section 2, we recall some useful preliminary heat kernel estimates and define our noise and notion of solution, which we take from Walsh \cite{Walsh86}. In Section 3 we prove Theorem \ref{thm: main result}, beginning with careful analysis of the three-distance function $F$ and its relation to Riemannian geometry. We then use this to perform a modified version of the iteration introduced in \cite{ChenThesis,ChenDalangAOP,ChenKim19}, where the modifications are needed due to geometry. In Section 4, we prove comparison principles for the equation, which we use to prove Theorem \ref{intro: thm: lower bound}, thus extending that result from \cite[Section 5]{CO25} as well.

\subsection{Notations}\label{notation}

To finish the introduction, we list some conventions and notation which we will employ for the rest of the paper.
\begin{itemize}
    \item We follow the usual convention and use $C_1,C_2,C_3$  and $c_1,c_2$, etc.\ to denote generic constants that are independent of quantities of interest. The exact values of these constants may change from line to line.
    \item For $t>0,x,y\in M$, $P_t(x,y)$ will denote the heat kernel on $M$.
    \item For $x\in M$ and $r>0$, $B(x,r)$ will denote the geodesic ball of radius $r$ centered at $x$.
    \item For $x,y\in M$, $\bm{d}(x,y)$ will denote the distance between the two points (under the Riemannian metric), and similarly for $\bm{d}(A,B)$ for any two Borel sets $A,B$.
    \item Let $m_0=m(M)=\Vol(M)$ be the volume of the manifold $M$ (which is finite because $M$ is compact).
    \item $D_M:=\max_{x,y\in M} \bm{d}(x,y)$ will denote the diameter of $M$, which is finite since $M$ is compact.
  
    \item For any measurable $f:M\to \R$, $\int_{M}f(x)m(dx)$ will denote integration with respect to the volume measure on $M$.
    \item $B_{\R^d}(r)$ will be a ball of radius $r>0$ centered at the origin in $\R^d$.
    \item For any measure $\mu$ on $M$ and $t>0,x\in M$, $J_\mu(t,x):=\int_M P_t(x,y) \mu(dy)$ will denote the homogeneous solution to the heat equation starting from $\mu$.
\end{itemize}

\section{Preliminaries}
In this section, we first present several preliminary results concerning the heat kernel and its upper bounds, which play an essential role in our analysis. We then briefly recall the definition of colored noise on compact manifolds introduced in \cite{COV23, CO25}. These noises are intrinsic to Riemannian manifolds and are smoother than white noise, which allows us to study \eqref{eq: SHE} in the It\^{o} sense.

\subsection{Upper Bounds of the heat kernel}

In contrast to the setting of \cite{CO25}, we allow the possibility that there exist infinitely many distance-minimizing geodesics connecting $x$ and $y$ when the two points are far apart. As a consequence, the measure induced by the Brownian bridge density $P_{t,x,y}(s,z)$ may concentrate on a submanifold of positive dimension. In this regime, standard Li--Yau type heat kernel estimates are no longer sharp enough to capture the correct quantitative behavior of $P_{t,x,y}(s,z)$ for small $t$. For this reason, we employ two different types of heat kernel estimates: one adapted to the case when $x$ and $y$ are close to each other (the Li-Yau bound), and another tailored to the case when they are far apart (see Lemma \ref{lem: heat kernel bd} below).

\begin{definition}\label{def: LY and Hsu Gaussians}
    For $\epsilon\geq0$, let $c_\epsilon:=(2+\epsilon)$ and define two Gaussian-type functions \begin{equation}\label{eq: LY and Hsu Gaussian def}
        G^\epsilon_t(r):=t^{-\frac{d}{2}}\exp\left(-\frac{r^2}{c_\epsilon t}\right),\quad\text{and}\quad \tilde{G}^\epsilon_t(r):=(t^{-\frac{d-1}{2}}\vee 1)G^\epsilon_t(r).
    \end{equation}
\end{definition}

The following Li-Yau type heat kernel estimate is standard \cite{LiYau86}, and will be used frequently in the sequel.
\begin{proposition}[Li-Yau Bound]\label{Prop:LiYau}
    For any $\epsilon>0$, $m\geq 1$ and $t>0$, we have $$P_t(x,y)\leq C(G^\epsilon_t(\bm{d}(x,y))+t^m\wedge1),\quad\mathrm{for\ all}\ x,y\in M.$$
\end{proposition}
For convenience, we will set $m=1$ for the rest of the paper. When $x$ and $y$ are far apart, the following global heat kernel estimate is sharper for small $t$.

\begin{lemma}\label{lem: heat kernel bd}
    There exists a constant $C>0$ such that  for all $x, y\in M$,
    \begin{align*}
    P_t(x,y)\leq \left\{
    \begin{array}{ll}
    C\tilde{G}^0_t(\bm{d}(x,y)),\quad&{0<t\leq1};\\
    C, & t>1.
    \end{array}
    \right. 
    \end{align*}

\end{lemma}

\begin{proof}
    For $0<t\leq 1$, \cite{EltonBBridge} and \cite[Theorem 5.3.4]{EltonBook} gives $$P_t(x,y)\leq C\tilde{G}^0_t(\bm{d}(x,y)).$$
     The classic spectral gap for the Laplace-Beltrami operator on compact Riemannian manifolds (see for example \cite[Chapter 3.2]{jost2005riemannian}) gives that there is $C>0$ depending only on $M$ such that, for any $t>1$, $$\sup_{x,y\in M}|P_t(x,y)-m_0^{-1}|\leq C e^{-\lambda_1t}.$$
    The proof is thus completed. 
\end{proof}

\medskip
It is clear that the following elementary relations hold for $G$ and $\tilde{G}$.
\begin{itemize}
\item  For    any $t>0$ and $0\leq \epsilon'<\epsilon$, we have \begin{equation}\label{est: LY beats LY larger eps}
        G^{\epsilon'}_t(r)<G^\epsilon_t(r),\quad \mathrm{for \  all}\ r>0.
    \end{equation}

\item    For any $\epsilon>0$ and $t\in (0,1)$, we have\begin{equation}\label{est: Hsu beats LY same eps}
        G^\epsilon_t(r)<\tilde{G}^\epsilon_t(r),\quad\mathrm{for\ all}\ r>0.
    \end{equation}
\item    Furthermore, for any fixed  $r'>0$ there exists some constant $C>0$ depending on $r'$ such that for any $t\in (0,1)$, and $0\leq \epsilon'<\epsilon$ we have\begin{equation}\label{est: LY beat Hsu larger eps}
        \tilde{G}^{\epsilon'}_t(r)<CG^\epsilon_t(r),\quad\mathrm{for\ all}\ r\geq r'.
    \end{equation}

\end{itemize}
{
The above relations enable us to make the modified induction argument in Section \ref{Sec Ln bd}.}

\subsection{Colored Noise and Mild Solution on Compact Riemannian Manifolds}\label{sec: noise}
In this section, we briefly recall the construction of the \emph{colored noise} on compact manifolds introduced in \cite{COV23, CO25}. In particular, Proposition~\ref{Prop: G_alpha} below characterizes the regularity properties of these noises.

\medskip

Denote by $0=\lambda_0<\lambda_1\leq\lambda_2\leq\dots$ the eigenvalues of $-\triangle_M$ and by $\phi_0, \phi_1,\phi_2,\dots$ an orthonormal sequence of corresponding eigenfunctions.
For any $\varphi\in L^2(M)$, there
is a unique decomposition
\begin{align}\label{E: func decomp}
  \varphi(x) = \sum_{n\geq0}a_n\phi_n(x).
\end{align}
In particular, $a_0={m_0}^{-1/2}\int_M\varphi  dm$ where $m_0=m(M)$ is the volume of $M$.

\medskip

A family of spatial Gaussian noises $\dot{W}$ on $M$ with parameters
$\alpha$, $\rho\ge 0$ can be constructed as follows. Let $(\Omega, \mathcal{F},\mathbb{P})$ be a
complete probability space such that for any $\varphi$ and $\psi$ n
$L^2(M)$, both ${W}\left(\varphi\right)$ and
${W}\left(\psi\right)$ are centered Gaussian random variables with
covariance given by
\begin{align}\label{E:NoiseCov}
 \mathbb{E} \left({W}\left(\varphi\right){W}\left(\psi\right)\right)
    = \langle\varphi,\psi\rangle_{\alpha,\rho} :=\rho a_0 {b}_0+\sum_{n\not=0} \frac{a_n{b}_n}{\lambda_n^{\alpha}}
   \end{align}
where $a_n$'s and $b_n$'s are the coefficients of $\varphi$ and $\psi$ in decomposition \eqref{E: func decomp},
respectively. For $\rho>0$, let $\mathcal{H}^{\alpha,\rho}$ be the completion of
$L^2(M)$ under $\langle\cdot,\cdot\rangle_{\alpha,\rho}$. It is clear that
$(\Omega, \mathcal{H}^{\alpha,\rho},\mathbb{P})$ gives an abstract Wiener space. When $\rho=0$, some special care is needed in order to identify a suitable Hilbert space $\mathcal{H}^{\alpha,0}.$ We refer readers to \cite{CO25} for more details in this special case.   It is clear from~\eqref{E:NoiseCov} that $L^2(M)\subset
  \mathcal{H}^{\alpha,\rho}\subset\mathcal{H}^{\beta,\rho}$ for
  $0\leq\alpha<\beta$. Moreover, the colored noise includes the white noise on $M$ if we pick $\rho=1$ and $\alpha=0$.

\medskip
The covariance structure $\langle\cdot,\cdot\rangle_{\alpha,\rho}$ admits a kernel. Indeed, let $P_t(x,y)$ be the heat kernel on $M$ and set for $\alpha,\rho>0$, 
\begin{align}\label{def: G_alpha etc}\bm{G}_{\alpha}(x,y):=\frac{1}{\Gamma(\alpha)}\int_0^\infty t^{\alpha-1}\left(P_t(x,y)-\frac{1}{m_0}\right) dt,\quad \textnormal{and}\ \ \bm{G}_{\alpha,\rho}(x,y):=\frac{\rho}{m_0}+\bm{G}_{\alpha}(x,y).
\end{align}

It is easy to see that one has
\begin{align*}
\langle\varphi,\psi\rangle_{\alpha,\rho}=\iint_{M^2}\phi(x)\bm{G}_{\alpha,\rho}(x,y)\psi(y)m(dx)m(dy).
\end{align*}

    Note that $\bm{G}_{\alpha}$ is the analogue of the Riesz kernel on $\mathbb{R}^d$. By \eqref{def: G_alpha etc} one has $\int_M \bm{G}_{\alpha}(x,y)m(dy)=0$. Hence $\bm{G}_{\alpha}$ is not non-negative. However, it can be shown that $\bm{G}_{\alpha}$ is bounded below on $M$ (see \cite{COV23} for example). One therefore can always pick a large enough $\rho$ so that the spatial covariance function $\bm{G}_{\alpha,\rho}$ is non-negative.

The following proposition gives the regularity of $\bm{G}_{\alpha}$ (hence $\bm{G}_{\alpha,\rho}$ as well) on diagonal. 

\begin{proposition}\label{Prop: G_alpha}
    For any $\alpha> 0$, we have $$\abs{\bm{G}_{\alpha}(x,y)}\leq \begin{cases}
        C_\alpha,& \alpha>d/2\\
        C_\alpha(1+\log^-\bm{d}(x,y)),& \alpha=d/2\\
        C_\alpha \bm{d}(x,y)^{2\alpha-d},& \alpha<d/2,
    \end{cases}$$
    where $\log^-(z)=\max(z,-\log z)$ and $\bm{d}(x,y)$ is the Riemannian distance on $M$.
\end{proposition}
\begin{proof}
    See \cite{Brosamler}.
\end{proof}

To close the discussion in this section, we define the noise on $\R_+\times M$ that is white in time and colored in space. 
\begin{definition}\label{def: Colored Noise}
    Let $\alpha>0$ and consider the following Hilbert space of space-time functions, 
\begin{align}\label{eq-hilb}
\mathcal{H}_{\alpha,\rho}=L^2(\mathbb{R}_+,\, \mathcal{H}^{\alpha,\rho}).  
\end{align}
On a complete probability space $(\Omega, \mathcal{F},\mathbb{P})$ we define a centered Gaussian family $\{W_{\alpha,\rho}(\phi); \phi\in L^2(\R_+)\cap\mathcal{H}_{\alpha,\rho}(M)\}$, whose covariance is given by 
\[
\mathbf E\left[ W_{\alpha,\rho}(\varphi) W_{\alpha,\rho}(\psi)\right]
=\int_{\R_+}\ \left\langle \varphi (t,\cdot) , \psi (t,\cdot)\right\rangle_{\alpha,\rho}  dt ,
\]
for $\varphi$, $\psi$ in $\mathcal{H}_{\alpha,\rho}$ in the space variable. This family is called
colored noise on $M$ that is white in time.
\end{definition}

\medskip
To simplify notation, we will drop the indexes $\alpha$ and $\rho$ and use $W$ for $W_{\alpha,\rho}$ throughout the rest of the paper.

\medskip
With the above, the following definition of a solution to \eqref{eq: SHE} is standard and will be used throughout the paper. Let $\mathcal{B}$ be the Borel $\sigma-$algebra of $M$. For $A\in \mathcal{B}$, $t\geq 0$, define $W_t(A):=W(\1_{[0,t]}(s)\1_{A}(x)).$  Define the filtration $(\mathcal{F}_t)_{t\geq 0}$ by $$\mathcal{F}_t:=\sigma(W_s(A):0\leq s\leq t, A\in \mathcal{B})\vee \mathcal{N},$$ where $\mathcal{N}$ is the collection of $\P-$null sets of $\mathcal{F}$. 

\begin{definition}\label{def: SHE mild}
    A random field $\set{u(t,x)}_{t\geq 0,x\in M}$ is an It\^{o} mild solution to \eqref{eq: SHE} if all the following holds.
    \begin{enumerate}[label=(\roman*)]
        \item Every $u(t,x)$ is $\mathcal{F}_t-$measurable.
        \item $u(t,x)$ is jointly measurable with respect to $\mathcal{B}\otimes \mathcal{F}$, where $\mathcal{B}$ is the Borel $\sigma-$algebra on $\R_+\times M$
        \item For all $(t,x)\in (0,\infty)\times M$, we have $$\E \left[\int_0^t\iint_{M^2} \bm{G}_{\alpha,\rho}(z,z')P_{t-s}(x,z)u(s,z)P_{t-s}(x,z')u(s,z')m(dz)m(dz')ds\right]< \infty$$
        \item $u$ satisfies \eqref{pam:mild}.
    \end{enumerate}
\end{definition}

\section{Proof of Theorem \ref{thm: main result}}\label{Sec: proof of thm 1}
{
Proposition \ref{Prop: G_alpha} makes it obvious that smaller $\alpha$ corresponds to more singular noise. Since Theorem \ref{thm: main result} concerns ``how low $\alpha$ can go", we will only address the case $\alpha<\frac{d}{2}$ in this section, as the other cases are simply easier.}\\
Our strategy for establishing existence and uniqueness of the mild solution to equation \eqref{eq: SHE} with arbitrary measure-valued initial conditions follows the iteration procedure developed in \cite{ChenThesis,ChenDalangAOP} (see also \cite{CO25}). For the convenience of the reader and for the sake of completeness, we briefly summarize this approach.

\medskip
For functions $h,w:\R_+\times M^4 \to \R$, define the operator $\rhd$ by
\[
h\rhd w(t,x_0,x,x_0',x')
:=\int_0^t \iint_{M^2}\,h(t-s,z,x,z',x')\,w(s,x_0,z,x_0',z')\,\bm{G}_{\alpha,\rho}(z,z')m(dz)m(dz')ds.
\]
We then introduce the sequence $\{\mathcal{L}_n\}_{n\geq 0}$ recursively by
\begin{align}\label{def: L_n}
\mathcal{L}_n(t,x_0,x,x_0',x')
:=\begin{cases}
P_t(x_0,x)P_t(x_0',x'), & n=0,\\
\mathcal{L}_0\rhd \mathcal{L}_{n-1}(t,x_0,x,x_0',x'), & n>0.
\end{cases}
\end{align}

The role of $\mathcal{L}_n$ can be understood as follows. Define
\[\widetilde{J}_\mu(t,x,x'):=J_\mu(t,x)J_\mu(t,x'), \qquad g(t,x,x'):=\E[u(t,x)u(t,x')].\]
By It\^{o}'s isometry, we have\[g(t,x,x')=\widetilde{J}_\mu(t,x,x')+\beta^2\int_0^t  \iint_{M^2} \,P_{t-s}(x,z)P_{t-s}(x',z')\bm{G}_{\alpha,\rho}(z,z')\,g(s,z,z')m(dz)m(dz')ds.\]

Iterating this relation yields
\begin{equation}\label{iteration corelation}
g(t,x,x')
=\widetilde{J}_\mu(t,x,x')
+\beta^2\iint_{M^2}\mu(dz)\mu(dz')\sum_{n=0}^\infty \beta^{2n}
\mathcal{L}_n(t,x,z,x',z').
\end{equation}
The validity of this computation relies on the convergence of the series
\begin{equation}\label{K_beta}
\mathcal{K}_\beta(t,x,z,x',z')
:=\sum_{n=0}^\infty \beta^{2n}\mathcal{L}_n(t,x,z,x',z').
\end{equation}

As shown in \cite{ChenThesis,ChenDalangAOP}, the existence and uniqueness of a mild solution to \eqref{eq: SHE}, as well as moment estimates for the solution, hinge on obtaining suitable bounds on $\mathcal{L}_n$. Moreover, these works observe that $\mathcal{L}_n$ can be controlled inductively once a proper estimate for $\mathcal{L}_1$ is available (see \cite[Sections 4--5]{COV23} or \cite[Section 4]{CO25} for a more recent exposition). Indeed, $\mathcal{L}_1$ is a variant of the quantity in \eqref{intro: main quantity for iteratin}, whose analysis lies at the heart of the problem.

\medskip

The following theorem is the key ingredient for the success of the above iteration procedure.

\medskip

\begin{theorem}\label{thm: L1 bd}
Let $M$ be any compact Riemannian manifold. For any $t>0$, define for $0\leq s\leq t$
\[
k^1(s):=k^1_L(s)+k^1_S(s),
\]
where $k^1_L$ is defined in Lemma \ref{lem: large time L1 bd} and $k^1_S$ is defined in Lemma \ref{lem:short time L1 bound}.
If $W$ satisfies Dalang's condition $\alpha>{(d-2)}/{2}$, then
\[
k^1(s)\leq C\big(s^{\frac{2\alpha-d}{2}}+1\big).
\]
Moreover,
\begin{align}\label{eq: L1 bd}
\mathcal{L}_1(t,x,y,x',y')
\leq C \big[G^{[2]}_t(d(x,y))+1\big]
\big[G^{[2]}_t(d(x',y'))+1\big]
\int_0^t k^1(s)\,ds.
\end{align}
Here $G^{[2]}$  is introduced in \eqref{def: G^ns} below.
\end{theorem}

 The remainder of Section \ref{Sec: proof of thm 1} is devoted primarily to the proof of Theorem \ref{thm: L1 bd}. Once \eqref{eq: L1 bd} is established, estimates for $\mathcal{L}_n$ are derived in detail in Section \ref{Sec Ln bd}, and the proof of Theorem \ref{thm: main result} is completed in Section \ref{Sec: existence and moment bound}.

\subsection{Properties of the Three Distances Function $F$}\label{sec:geod and F}

{As discussed in the introduction, the analysis of the quantity
\eqref{intro: main quantity for iteratin} (or equivalently $\mathcal{L}_1$)
relies crucially on understanding the interplay among the three distance
functions in $F_{s,t;x,y}(z)$ (see \eqref{func: F}). Elementary considerations show that $F_{s,t;x,y}$ is always nonnegative and
attains its minimum value $0$ precisely along distance-minimizing geodesics
connecting $x$ and $y$. Consequently, the behavior of $F_{s,t;x,y}$ in a
neighborhood of its minimizers governs the concentration properties of the
Brownian bridge measure $P_{t,x,y}(s,z)$, especially in the small-time regime.
As one would expect, this behavior reflects both the local and global geometry
of the underlying manifold $M$. In this section, we deal with the geometry of geodesics in its full generality and  study the behavior of $F$
near its minimizers without imposing any curvature assumptions on $M$.

To this end, for $a\in(0,1)$ and $x,y\in M$, we introduce the function
$F_{a;x,y}:M\to\R_+$ defined by
\[
F_{a;x,y}(z)
:=(1-a)\bm{d}(x,z)^2+a\bm{d}(z,y)^2-a(1-a)\bm{d}(x,y)^2 ,
\]
and obtained by factoring $\frac{1}{s(t-s)/t}$ out of $F_{s,t;x,y}$ and letting $a=\frac{s}{t}$. Consideration of $F_{a;x,y}$ with $a=\frac{1}{2}$ has a long history in the study of small-time heat kernel asymptotics on Riemannian manifolds, especially on the cut locus, going back to \cite{Molchanov} and continuing through the recent work \cite{WithLudo}.

One of the key differences between the present work and \cite{CO25} is the following two
results, which provide effective control of $F_{a;x,y}$ when the points $x$
and $y$ are far apart, and close together, respectively. Note that neither result requires that $M$ is compact, and thus we state them in more generality, with an eye toward future work on the PAM on non-compact manifolds.

\begin{lemma}\label{lem: F concentration}
Let $M$ be any complete (Riemannian) manifold. For every $a\in(0,1)$ and all $x,y,z\in M$,
\begin{equation}\label{Lem1 F}
F_{a;x,y}(z)\geq \big(\bm{d}(x,z)-a\,\bm{d}(x,y)\big)^2.
\end{equation}
\end{lemma}

\begin{remark}
The lemma is motivated by an explicit computation in the case
$M=\mathbb{S}^{d-1}$. When $x$ and $y$ are antipodal points, the inequality
in \eqref{Lem1 F} becomes an equality. In this setting, the minimizers of
$F_{a;x,y}$ form a latitude sphere (depending on $a$), which is a submanifold
of codimension one. This configuration represents the ``worst-case'' scenario
for concentration, in contrast to situations where the set of minimizers
consists of finitely many points.
\end{remark}
}

\begin{proof}
To set up the proof, we introduce some notation. Let
\[
r= \bm{d}(x,z)-a\bm{d}(x,y) \quad\text{and}\quad \ell=(1-a)\bm{d}(x,y)-\bm{d}(y,z),
\]
with both considered as functions of $z$. (Geometrically, $r$ is the signed distance from the sphere of radius $ad(x,y)$ around $x$, with $r$ positive outside of the ball and negative inside, while $\ell$ is the signed distance from the sphere of radius $(1-a)\bm{d}(x,y)$ around $y$, with $\ell$ negative outside of the ball and positive inside.)  We think of them as the right and left defects from optimizing the energy, as measured by $F_{a;x,y}$. The triangle inequality gives
\[\begin{split}
\bm{d}(x,z)+\bm{d}(z,y)\geq \bm{d}(x,y) & \Rightarrow a\bm{d}(x,y)+r+(1-a)\bm{d}(x,y)-\ell \geq \bm{d}(x,y) \\
&\Rightarrow r-\ell \geq 0 ,
\end{split}\]
and $r=\ell$ exactly when $z$ lies on a minimal geodesic from $x$ to $y$. 

With a bit of algebra, we can write $F_{a;x,y}$ in terms of $r$ and $\ell$. To simplify the notation, we write $d$ for $\bm{d}(x,y)$ {(note that the dimension does not appear in this section)}. We have
\[\begin{split}
F_{a;x,y}(z) &= (1-a)\lb a^2d^2+2ard+r^2\rb +a\lb (1-a)^2d^2-2(1-a)\ell d+\ell^2\rb -a(1-a)d^2   \\
&= a^2(1-a)d^2 +a(1-a)^2d^2 -a(1-a)d^2 +a(1-a)\lp 2rd-2\ell d\rp +(1-a)r^2+a\ell^2 \\
&= 2a(1-a)d\lp r-\ell\rp +(1-a)r^2 +a\ell^2 .
\end{split}\]
(Recalling that $r-\ell \geq 0$, it's clear from this that $F_{a;x,y}$ is non-negative, and $F_{a;x,y}(z) =0$ exactly when $r=\ell=0$, which happens when $z$ is ``$a$ of the way'' along a minimal geodesic from $x$ to $y$. This explains the geometry of $F_{a;x,y}$ and the intuition behind $r$ and $\ell$.)

Moreover, the claim in the lemma is exactly that $F_{a;x,y}\geq r^2$, which means it's enough to show that $F_{a;x,y}-r^2\geq 0$ for all $z$. Continuing from the previous computation, we have
\[\begin{split}
F_{a;x,y}(z) -r^2&=  2a(1-a)d\lp r-\ell\rp +(1-a)r^2 +a\ell^2 -r^2 \\
&=  2a(1-a)d\lp r-\ell\rp +a\lp \ell^2-r^2 \rp \\
&= 2a(1-a)d\lp r-\ell\rp +a\lp \ell-r \rp\lp \ell+r\rp \\
&= a(r-\ell) \lb 2(1-a)d-(\ell+r)\rb .
\end{split}\]
Since $r-\ell \geq 0$, to show this last line is non-negative, we need to show that $2(1-a)d\geq \ell +r$. But the triangle inequality plus rewriting $\bm{d}(x,z)$ and $\bm{d}(y,z)$ in terms of $r$, $\ell$, and $d$ gives
\[\begin{split}
\bm{d}(x,z)-\bm{d}(y,z) \leq \bm{d}(x,y)  & \Rightarrow
ad+r -\lb (1-a)d-\ell\rb \leq d \\
 &\Rightarrow  2ad +r+\ell \leq 2d \\
  & \Rightarrow r+\ell \leq 2d(1-a) ,
\end{split}\]
exactly as desired.
\end{proof}

We need a better estimate for close points.
If $\bm{d}(x,y)$ is less than the injectivity radius of $M$, then there is a unique minimizing geodesic $\gamma_{xy}$ from $x$ to $y$, with the convention that $\gamma_{xy}$ has constant speed parametrization with $\gamma_{xy}(0)=x$ and $\gamma_{xy}(1)=y$. With this notation, we can state our next result.

\begin{lemma}\label{lem: F concentration no cutlocus}
Suppose M has positive injectivity radius (meaning it is bounded from below by a positive constant) and sectional curvature bounded from above; note that this always holds if $M$ is compact. Then we can find constants $\zeta\in(0,1)$ and $D>0$ depending only on these two bounds (and with $D$ less than the injectivity radius of $M$) such that, for any $x,y\in M$ with $\bm{d}(x,y)< D$,
\begin{equation}\label{Eqn:LemmaZeta}
F_{a;x,y}(z) \geq \zeta \bm{d}^2(\gamma_{xy}(a),z)
\end{equation}
for any $a\in[0,1]$ and any $z\in M$.
\end{lemma}

\begin{proof}
Because the injectivity radius is positive, we can assume that $D>0$ is less than the injectivity radius, so that for any $x$ and $y$ with $\bm{d}(x,y)<D$, there is a unique minimizing geodesic connecting them, and thus $\gamma_{xy}$ is well defined. (Even if the injectivity radius of $M$ is infinite, we choose some finite $D$ for convenience, to avoid multiple cases in what follows. But see Remark \ref{Rem:ImprovingZeta}.) By completeness, $\gamma_{xy}$ can be extended indefinitely in both directions, and in particular, $\gamma_{xy}(t)$ is defined for all $t\in \bR$. While $\gamma_{xy}$ might not be minimizing between all choices of $\gamma_{xy}(t)$ and $\gamma_{xy}(t')$, it will be minimizing if $t$ and $t'$ are sufficiently close. We assume that we always have $\bm{d}(x,y)<D$ for $D$ as above, and note that we are allowed to further reduce $D$ as necessary. 

We now proceed in several steps.

\emph{Step 1}: Suppose that $t$ is such that $\gamma_{xy}$ is the unique minimizing geodesic between $\gamma_{xy}(t)$ and $x$ and between $\gamma_{xy}(t)$ and $y$ (which will be the case for $t$ in some neighborhood of $[0,1]$, since $\bm{d}(x,y)$ is strictly less than the injectivity radius). Then, again letting $\bm{d}(x,y)=d$ for ease of notation and taking the parametrization of $\gamma_{xy}$ into account, we have that
\[\begin{split}
&\bm{d}\lp x,\gamma_{xy}(t)\rp = |td|, \quad \bm{d}\lp y,\gamma_{xy}(t)\rp = |td-d|, \\
&\quad \text{and}\quad
\bm{d}\lp \gamma_{xy}(a),\gamma_{xy}(t)\rp = |td-ad| .
\end{split}\]
We compute that
\[\begin{split}
F_{a;x,y}\lp \gamma_{xy}(t) \rp &= (1-a)(td)^2+a(td-d)^2 -a(1-a)d^2 \\
&= \lp (t-a)d\rp^2 \\
&= \bm{d}^2\lp \gamma_{xy}(a),\gamma_{xy}(t)\rp .
\end{split}\]
This means that for $z$ on $\gamma_{xy}$, close enough to the segment between $x$ and $y$, \eqref{Eqn:LemmaZeta} becomes an equality with $\zeta=1$.

\emph{Step 2}: From here, we break the argument into two cases, namely when the point $z$ is far from $\gamma_{xy}(a)$ and when $z$ is near $\gamma_{xy}(a)$.

We begin with far points. Because $\bm{d}\lp x,\gamma_{xy}(a)\rp = a\bm{d}(x,y)$ and $\bm{d}\lp \gamma_{xy}(a),y\rp = (1-a)\bm{d}(x,y)$, the triangle inequality gives that
\[
\bm{d}(x,z)\geq \bm{d}\lp \gamma_{xy}(a),z \rp - a\bm{d}(x,y) \quad\text{and}\quad \bm{d}(z,y)\geq \bm{d}\lp \gamma_{xy}(a),z \rp - (1-a)\bm{d}(x,y) .
\]
Assuming that $\bm{d}\lp \gamma_{xy}(a),z \rp>D$, the right-hand side of both of these inequalities will be non-negative for any $a\in[0,1]$, and thus we can square both inequalities to get lower bounds for $\bm{d}(x,z)$ and $\bm{d}(z,y)$.
Using these lower bounds then gives a lower bound for $F$ as
\[\begin{split}
F_{a;x,y}(z) &\geq (1-a)\lb  \bm{d}\lp \gamma_{xy}(a),z \rp - a\bm{d}(x,y) \rb^2 +a\lb \bm{d}\lp \gamma_{xy}(a),z \rp  - (1-a)\bm{d}(x,y) \rb^2\\
&\quad -a(1-a)\bm{d}^2(x,y) \\
&= \bm{d}^2\lp \gamma_{xy}(a),z \rp-4a(1-a)\bm{d}(x,y)\cdot \bm{d}\lp \gamma_{xy}(a),z \rp ,
\end{split}\]
where we've omitted the details of simplifying the right-hand side. Since the supremum of $a(1-a)$ for $0\leq a\leq 1$ is $\frac{1}{4}$, we have that
\[
F_{a;x,y}(z) \geq \bm{d}^2\lp \gamma_{xy}(a),z \rp-\bm{d}(x,y)\cdot \bm{d}\lp \gamma_{xy}(a),z \rp .
\]
Since we assume that $\bm{d}(x,y)<D$, this shows that, for any $\zeta\in(0,1)$, we have
\[
F_{a;x,y}(z) \geq \zeta \bm{d}^2\lp \gamma_{xy}(a),z \rp
\]
whenever $\bm{d}\lp \gamma_{xy}(a),z \rp \geq \frac{D}{1-\zeta}$, for any $a\in[0,1]$. Note that $\frac{D}{1-\zeta}>D$, so this condition implies our earlier requirement that $\bm{d}\lp \gamma_{xy}(a),z \rp>D$. For future use, we rephrase the conclusion as follows: for any $\eps>0$, we can find $D>0$ and $\zeta_f \in(0,1)$, depending only on $\eps$ and the injectivity radius of $M$, such that for any $x$ and $y$ with $\bm{d}(x,y)<D$,
\[
F_{a;x,y}(z) \geq \zeta_f \bm{d}^2\lp \gamma_{xy}(a),z \rp
\]
whenever $\bm{d}\lp \gamma_{xy}(a),z \rp \geq \eps$, for any $a\in[0,1]$.

\emph{Step 3}: The situation for $z$ near $\gamma_{xy}(a)$ is more complicated. We let $\overline{\gamma}_{xy}$ denote the segment of $\gamma_{xy}$ from $x$ to $y$, so $\{\gamma_{xy}(a):  0\leq a\leq 1\}$. We will rely on comparison with spheres, and thus we first consider $F$ on spheres themselves. In particular, for $K>0$, let $\bS^n(1/\sqrt{K})$ be the sphere of dimension $n$ and constant curvature $K$. Then $\bS^n(1/\sqrt{K})$ has injectivity radius and diameter both equal to $\frac{\pi}{\sqrt{K}}$. In particular, we have (exponential) polar coordinates $(r,\theta)\in[0,\frac{\pi}{\sqrt{K}})\times \bS^{n-1}$ around any point $x$. Suppose that $y$ is such that $\bm{d}(x,y)\leq \frac{\pi}{4\sqrt{K}}$. Then there is some $\theta_0\in \bS^{n-1}$ such that $\overline{\gamma}_{xy}$ corresponds to $[0,\bm{d}(x,y)]\times\theta_0$ in polar coordinates around $x$.

We wish to understand $F_{a;x,y}$ near $\overline{\gamma}_{xy}$. From above, we already know the value of $F_{a;x,y}$ on $\overline{\gamma}_{xy}$, so we consider a tubular neighborhood of some small radius $\eps>0$ and the change of $F_{a;x,y}$ as we move away from $\overline{\gamma}_{xy}$ along a geodesic perpendicular to $\gamma_{xy}$. We identify this tubular neighborhood with the normal bundle $(r,\rho,\phi)\in [0,\bm{d}(x,y)]\times [0,\eps)\times \bS^{n-2}$ via the exponential map on fibers, where we've put polar coordinates on the normal fibers. We use $\pi(z)$ for the projection of $z$ onto the base point of the fiber, namely the point on $\overline{\gamma}_{xy}$ which $z$ is over. Let $v_1,\ldots, v_{n-1}$ be a smooth orthonormal frame tangent to the normal fibers. Then $\nabla_{v_i} \bm{d}(x,z)|_{\rho=0}=0$ for all $i$, because $v_i$ is perpendicular to $\gamma_{xy}$. Further, it is a standard computation, starting from the explicit expression for the metric on $\bS^n(1/\sqrt{K})$ in polar coordinates, that
\[
\frac{1}{2}\Hess_{v_i,v_i} \lp \bm{d}^2(x,z)\rp|_{\rho=0}= \bm{d}(x,z)\sqrt{K}\cot\lp \bm{d}(x,z)\sqrt{K}\rp \quad \text{for $r\in [0,\pi/\sqrt{K})$ and any $i$,}
\]
where $x$ is fixed and the Hessian acts on $z$. Since $r$ is the distance from $x$ when $\rho=0$, by definition, if we write $z$ in normal bundle coordinates as $(r,\rho,\phi)$, expanding in $\rho$ gives that
\[
\bm{d}^2(x,z) = r^2 + r\sqrt{K}\cot\lp r\sqrt{K}\rp \rho^2 +O(\rho^3) 
\]
for $z$ in the tubular neighborhood. Note that by rotational symmetry, $\bm{d}(x,z)$ doesn't depend on $\phi$, so it's enough to consider $r$ and $\rho$. Also, we have that $x\cot x \geq \pi/2$ for $x\in[0,\pi/4]$, and since we consider only $z$ in the tubular neighborhood, which means that $r\sqrt{K}\leq\pi/4$, we have that
\[
\bm{d}^2(x,z) = r^2 + \frac{\pi}{2} \rho^2 +O(\rho^3) 
\]
for all $z=(r,\rho,\phi)$ in the tubular neighborhood.

Of course, the role of $x$ and $y$ is symmetric, so we have the same estimate for $\bm{d}^2(y,z)$. More precisely, for $z$ in the tubular neighborhood, we let $r'$ be the distance from $y$ to $\pi(z)$, and we have
\[
\bm{d}^2(y,z) = (r')^2 + \frac{\pi}{2} \rho^2 +O(\rho^3) .
\]
Note that $r'=\bm{d}(x,y)-r$. Because $\bm{d}^2(x,\cdot)$ and $\bm{d}^2(y,\cdot)$ are both smooth on the closure of our tubular neighborhood, the $O(\rho^3)$ term is uniformly controlled, and thus we have that, if $\eps$ is small enough
\begin{equation}\label{Eqn:Step2WithR}
(1-a)\bm{d}^2(x,z)+a\bm{d}^2(z,y)\geq (1-a)r^2 +a\lp \bm{d}(x,y)-r\rp^2+\frac{\pi}{3} \rho^2 
\end{equation}
for all $r\in[0,\bm{d}(x,y)]$, $a\in [0,1]$, and $\rho\in[0,\eps]$. From Step 1, we know that 
\[
(1-a)r^2 +a\lp \bm{d}(x,y)-r\rp^2 -a(1-a)\bm{d}^2(x,y) = F_{a;x,y}\lp \pi(z)\rp  = \bm{d}^2\lp \gamma_{xy}(a),\pi(z)\rp ,
\]
and so \eqref{Eqn:Step2WithR} implies that $F_{a;x,y}(z) \geq \bm{d}^2\lp \gamma_{xy}(a),\pi(z)\rp +\frac{\pi}{3} \rho^2 $.

Finally, note that, by continuity and the fact that $x$ any $y$ are well away from the cut locus, we can allow $\pi(z)$ to be a point on $\overline{\gamma}_{xy}$ extended by $\eps$ in either direction, at the cost of possibly shrinking $\eps$ and reducing the $\frac{\pi}{3}$. This gives a type of ``tubular $\eps$-neighborhood;'' namely, we let $T_{\eps}\lp\gamma_{xy}\rp$ be given by $(r,\rho,\phi)\in (-\eps,\bm{d}(x,y)+\eps) \times [0,\eps)\times \bS^{n-2}$. Since the exact value of our positive constants doesn't matter to us here, we summarize this by stating that there are $\eps>0$ and $C>0$, depending only on $K>0$, such that for any $x,y\in \bS^n(1/\sqrt{K})$ with $\bm{d}(x,y)\leq \frac{\pi}{4\sqrt{K}}$, $T_{\eps}\lp\gamma_{xy}\rp$ admits smooth exponential coordinates as above and 
\[
F_{a;x,y}(z) \geq \bm{d}^2\lp \gamma_{xy}(a),\pi(z)\rp +C \rho^2 
\]
for all $r\in(-\eps,\bm{d}(x,y)+\eps)$, $a\in [0,1]$, and $\rho\in[0,\eps)$, where $z=(r,\rho,\phi)$ in these exponential coordinates. Also note that  $T_{\eps}\lp\gamma_{xy}\rp$ has the virtue that it contains the ball $B_{\eps}(\gamma_{xy}(t))$ for any $t\in[0,1]$. This is the estimate we need for comparison spheres.

\emph{Step 4}: We now consider points $z$ near $\overline{\gamma}_{xy}$ on a general $M$ as in the lemma. We let $K>0$ be an upper bound on the sectional curvature of $M$ (so if $M$ has non-positive curvature, we take, say $K=1$ for convenience, similar to the way we treated $D$ at the beginning of the proof). We also assume that $D>0$ is less than $\frac{\pi}{4\sqrt{K}}$ and less than half the injectivity radius of $M$. Then for small enough $\eps$ (depending only on $K$ and $D$), for any $x$ and $y$ with $\bm{d}(x,y)<D$, the tubular $\eps$-neighborhood $T_{\eps}\lp\gamma_{xy}\rp$ can be defined just as in the previous step, so that it admits exponential coordinates with $\rho$ again giving the distance from the extension of $\overline{\gamma}_{xy}$. In what follows, we abbreviate $T_{\eps}\lp\gamma_{xy}\rp$ as $T_{\eps}$.

The point is that we have hinge comparison with respect to $\bS^n(1/\sqrt{K})$ on $T_{\eps}$. In particular, consider any $z$ in $T_{\eps}$, with associated projection $\pi(z)$ and coordinate $\rho$. Then $\bm{d}(x,z)$ is the length of the hypotenuse of a right triangle with legs of length $\bm{d}(x,\pi(z))$ and $\rho$. We have chosen $T_{\eps}$ small enough so that we have comparison for an upper curvature bound; see \cite[Theorem 4.1]{Karcher}. Thus $\bm{d}(x,z)$ is greater than or equal to the length of the hypotenuse of the comparison hinge in $\bS^n(1/\sqrt{K})$. This is what was bounded in the previous step. Moreover, we can apply the same argument to bound $\bm{d}(y,z)$ from below. In light of the definition of $F_{a;x,y}(z)$ and the results of Step 3, we therefore have that there are 
$\eps>0$ and $C>0$, depending only on $K$ and $D$ as above, such that for any $x,y\in M$ with $\bm{d}(x,y)\leq D$, $T_{\eps}$ admits smooth exponential coordinates and 
\[
F_{a;x,y}(z) \geq \bm{d}^2\lp \gamma_{xy}(a),\pi(z)\rp +C \rho^2 
\]
for all $z\in T_{\eps}$ and $a\in [0,1]$.

We now need to bound $\bm{d}^2\lp \gamma_{xy}(a),z \rp$ from above for $z\in T_{\eps}$. Because we avoid assuming a lower curvature bound, we can't do this via comparison hinges. Instead, if we let $u=\bm{d}\lp \gamma_{xy}(a),\pi(z)\rp$ for ease of notation, we note that the triangle inequality gives that $\bm{d}\lp \gamma_{xy}(a),z \rp \leq u+\rho$. With this notation, we have that, for any $\zeta>0$,
\[
F_{a;x,y}(z) - \zeta \bm{d}^2\lp \gamma_{xy}(a),z \rp  \geq u^2 +C\rho^2 -\zeta(u+\rho)^2 .
\]
Since $(u+p)^2 \leq 4\max\{u^2,\rho^2\}$, we can find $\zeta>0$ depending only on $C$ that makes the right-hand side of the above non-negative. But that is equivalent to the inequality we are trying to prove. We summarize this by observing that there exist $\eps>0$ and $\zeta_c>0$, depending only on $K$ and $D$ as above, such that for any $x,y\in M$ with $\bm{d}(x,y)\leq D$,
\[
F_{a;x,y}(z) \geq \zeta_c \bm{d}^2\lp \gamma_{xy}(a),z \rp 
\]
for all $z\in T_{\eps}$ and $a\in [0,1]$.

\emph{Step 5}: We can now finish the proof. As in the previous step, we have $K$ and $D$ depending only on the upper curvature bound and the injectivity radius of $M$. Then the previous step says we can find $\zeta_c>0$ and $\eps>0$ depending only on $K$ and $D$ such that, if $\bm{d}(x,y)<D$, then 
\[
F_{a;x,y}(z) \geq \zeta_c \bm{d}^2\lp \gamma_{xy}(a),z \rp 
\]
for all $z\in T_{\eps}$ and $a\in [0,1]$. For this same $\eps$, Step 2 implies we can find $D'>0$ and $\zeta_f\in (0,1)$
depending only on $\eps$ and the injectivity radius of $M$, such that for any $x$ and $y$ with $\bm{d}(x,y)<D$,
\[
F_{a;x,y}(z) \geq \zeta_f \bm{d}^2\lp \gamma_{xy}(a),z \rp
\]
whenever $\bm{d}\lp \gamma_{xy}(a),z \rp \geq \eps$, for any $a\in[0,1]$.

As noted above, the definition of $T_{\eps}$ implies that any $z\not\in T_{\eps}$ satisfies $\bm{d}\lp \gamma_{xy}(a),z \rp \geq \eps$, and thus all $z\in M$ satisfy one of these two conditions. If $D'<D$, we can replace $D$ with $D'$, and we can take $\zeta=\min\{\zeta_c,\zeta_f\}$. This produces a $D>0$ and a $\zeta\in(0,1)$ that, looking back at the dependencies just described, depend only on the upper curvature bound and the injectivity radius of $M$, such that if $\bm{d}(x,y)<D$,
\[
F_{a;x,y}(z) \geq \zeta \bm{d}^2\lp \gamma_{xy}(a),z \rp
\]
for all $z\in M$ and any $a\in[0,1]$. We also recall $D$ will be less than the injectivity radius, as it has been throughout the proof, so that $\gamma_{xy}$ and $\gamma_{xy}(a)$ are well defined.

\end{proof}

\begin{remark}\label{Rem:ImprovingZeta}
As we see from the proof, considering points on $\gamma_{xy}$ shows that there is no $M$ for which $\zeta$ can be taken to be bigger than 1, no matter how small $D$ is.

On the other hand, when $M=\bR^d$, direct calculation, just with the Pythagorean theorem, shows that
\[
F= \bm{d}^2\lp \gamma_{xy}(a),z\rp
\] 
for all $x$, $y$, and $z$ in $\bR^d$ and all $a\in[0,1]$. In particular, for $\bR^d$, the lemma holds with $\zeta=1$ and $D=\infty$, and this is sharp, even at individual points. As this indicates, not only could one be more explicit about $\zeta$ and $D$, as functions of the injectivity radius and curvature bound, in Lemma \ref{lem: F concentration no cutlocus}, but one expects a better version of Lemma \ref{lem: F concentration no cutlocus} for various classes of manifolds, such as manifolds with lower curvature bounds or Cartan-Hadamard manifolds. Nonetheless, that degree of precision is not needed for our purposes here, and we don't pursue it.
\end{remark}

\subsection{Some Useful Integral Estimates}\label{sec: useful int est}

In this section, we establish several key integral estimates that will be needed later.
Recall that the density of a Brownian bridge admits the representation
\[
P_{t,x,y}(s,z)
=\frac{P_s(x,z)P_{t-s}(z,y)}{P_t(x,y)}.
\]
For any fixed $\epsilon>0$, and for our purposes, the following quantity effectively
captures the behavior of $P_{t,x,y}(s,z)$ for $0<t\leq 1$:
\begin{equation}\label{def: quotients of different UB}
G^\epsilon_{t,x,y}(s,z)
:=\frac{
G^\epsilon_{s}(\bm d(x,z))\,G^\epsilon_{t-s}(\bm d(z,y))
}{
\1_{\bm d(x,y)< D}\,G^\epsilon_t(\bm d(x,y))
+\1_{\bm d(x,y)\geq D}\,\tilde{G}^\epsilon_t(\bm d(x,y))
}.
\end{equation}
Here $D>0$ is taken from the statement of Lemma \ref{lem: F concentration no cutlocus}. 

\medskip

As can be seen from \eqref{def: quotients of different UB}, we employ a sharper
heat kernel estimate in the denominator when the points $x$ and $y$ are far
apart. This refinement allows us to capture more accurately the small-time
behavior of the Brownian bridge density $P_{t,x,y}(s,z)$ as $t\downarrow 0$.
In this regime, there may exist infinitely many distance-minimizing geodesics
connecting $x$ and $y$, and the Brownian bridge measure can concentrate on a
submanifold of positive dimension.

\medskip

This distinction is crucial for establishing estimates 
\eqref{est: bridge density} and \eqref{est: bridge density with riesz} below.
It constitutes a key conceptual difference between the present work and
\cite{CO25}, as well as earlier studies on the well-posedness of
\eqref{eq: SHE} with measure-valued initial data.

\medskip
Since we must distinguish between the cases where $x$ and $y$ are close and where they
are far apart, it is essential—particularly for the iteration procedure described at
the beginning of Section~\ref{Sec: proof of thm 1}—to unify the resulting bounds.
To this end, we make use of the comparison estimates
\eqref{est: LY beats LY larger eps}--\eqref{est: LY beat Hsu larger eps},
which allow us to absorb both regimes into a single estimate.
This is the reason for introducing the additional tuning parameter $\epsilon$ in the
definition of $G^\epsilon_{t,x,y}$.

\medskip
The following integral estimates will be useful for proving Theorem \ref{thm: L1 bd}.
\begin{lemma}\label{lem: single integral bounds} Fix any $\epsilon>0$. Suppose  $0<a<\frac{1}{2}$ and $0<\alpha<\frac{d}{2}$. The following inequalities hold for any $t>0$ and $0<\epsilon'\leq\epsilon$, with the constants on the right hand side possibly depending on $\epsilon$ if the left hand side involves $\epsilon'$.
\begin{equation}\label{est:1pt Gauss est}
     {  \sup_{x\in M}} \int_{M} G^{\epsilon'}_t(\bm d(x,y)) m(dy)\leq C,
    \end{equation}
    \begin{equation}\label{est: Riesz}
    { \sup_{x\in M}}   \int_{M} \bm d(x,y)^{2\alpha-d} m(dy)\leq C,
    \end{equation}
    \begin{equation}\label{est: pt gauss riesz 1 int}
     { \sup_{x,x'\in M} } \int_{M} G^{\epsilon'}_t(\bm d(x,y)) \bm d(x',y)^{2\alpha-d} m(dy)\leq C(t^{\frac{2\alpha-d}{2}}+1).
    \end{equation}
    In addition, for all $0<t\leq 1$ and {all $a\in(0,\frac{1}{2})$},
    \begin{equation}\label{est: bridge density}
    {\sup_{x,x'\in M}}    \int_{M} G^{\epsilon'}_{t,x,x'}(at,y) m(dy)\leq C,
    \end{equation}
    \begin{equation}\label{est: bridge density with riesz}
    {\sup_{x,x',x_0\in M}}  \int_{M} {G^{\epsilon'}_{t,x,x'}(at,y)}\bm d(y,x_0)^{2\alpha-d} m(dy)\leq C\left[\left(a(1-a)t\right)^{\frac{2\alpha-d}{2}}+1\right].
    \end{equation}
\end{lemma}
\begin{proof}
    We fix $\delta=\frac{i_M}{16}$.
        In the rest of the proof, we will use $$\int_M G^{\epsilon'}_t(\delta) m(dy)= G^{\epsilon'}_t(\delta)m(M)\leq G_t^{\epsilon}(\delta)m(M)\leq C_{\epsilon,M}$$
    without reference. For \eqref{est:1pt Gauss est}, simply decompose $M=B(x,\delta)\sqcup B(x,\delta)^c$ and note that for $y\in B(x,\delta)^c$ we have $d(x,y)>\delta$, which gives us \begin{align*}
        &\int_M G^{\epsilon'}_t(\bm d(x,y))m(dy)= \left(\int_{B(x,\delta)}+\int_{B(x,\delta)^c}\right)G^{\epsilon'}_t(\bm d(x,y))m(dy)\\
        \leq& C\left(\int_{B^{\R^d}(\delta)} G^{\epsilon'}_t(|y|) dy+\int_M G^{\epsilon'}_t(\delta) m(dy)\right)\\
        &\leq C[(c_{\epsilon'})^{\frac{d}{2}}+C_{\epsilon,M}]\leq C_{\epsilon,M}.
    \end{align*}
    For \eqref{est: Riesz}, only elementary arguments are needed after applying the same decomposition.\\
    For \eqref{est: pt gauss riesz 1 int}, we will first show \[\int_{M} G^{\epsilon'}_t(\bm d(x,y)) \bm d(x',y)^{2\alpha-d} m(dy)\leq 2\sup_{z\in M}\int_{M} G^{\epsilon'}_t(\bm d(z,y)) \bm d(z,y)^{2\alpha-d} m(dy).\]
    For $x=x'$, it is obvious, so it remains to address $x\neq x'$.
    Observe that $M=M_x\cup M_{x'}$, where $M_x:=\set{y\in M: \bm d(x,y)\leq \bm d(x',y)}$ and $M_{x'}$ is defined analogously. They obviously have measure 0 intersection
        and the fact that they are closed follows from the continuity of $f_{x,x'}(y):=\bm d(x,y)-\bm d(x',y)$. Since $e^{-\frac{r^2}{c_\epsilon t}}$ and $r^{2\alpha-d}$ are both decreasing in $r$ for $r>0$, we have \begin{align*}
        &\int_{M} G^{\epsilon'}_t(\bm d(x,y))\bm d(x',y)^{2\alpha-d} m(dy)=\left(\int_{M_x}+\int_{M_{x'}}\right)G^{\epsilon'}_t(\bm d(x,y))\bm d(x',y)^{2\alpha-d} m(dy)\\
        \leq& \int_{M_x} G^{\epsilon'}_t(\bm d(x,y)) \bm d(x,y)^{2\alpha-d} m(dy)+ \int_{M_{x'}}G^{\epsilon'}_t(\bm d(x',y)) \bm d(x',y)^{2\alpha-d} m(dy)\\
        \leq& 2\sup_{z\in M}\int_M G^{\epsilon'}_t(\bm d(z,y)) \bm d(z,y)^{2\alpha-d} m(dy),\qquad \text{as desired}.
    \end{align*}
    Now for any $z\in M$, we again use $M=B(z,\delta)\cup B(z,\delta)^c$  and thus \begin{align*}
        &\int_M G^{\epsilon'}_t(\bm d(z,y))\bm d(z,y)^{2\alpha-d} m(dy)\\
        \leq &\int_{B(z,\delta)} G^{\epsilon'}_t(\bm d(z,y)) \bm d(z,y)^{2\alpha-d} m(dy)+\int_{B(z,\delta)^c} G^{\epsilon'}_t(\delta) \delta^{2\alpha-d} m(dy)
        \\
        \leq &C\int_{\R^d} G^{\epsilon'}_t(|y|) |y|^{2\alpha-d} dy+C_{\epsilon,M}\\
        =&C(c_{\epsilon'})^{\alpha-1}t^{\frac{2\alpha-d}{2}}+C_{\epsilon,M}\leq C_{\epsilon,\alpha,M}(t^{\frac{2\alpha-d}{2}}+1),\qquad \text{as desired}.
    \end{align*}
    For \eqref{est: bridge density} and \eqref{est: bridge density with riesz}, the nature of the denominator in $G^{\epsilon'}_{t,x,x'}(at,y)$ requires us to deal with the cases $\bm d(x,x')<D$ and $\bm d(x,x')\geq D$ separately. The case $\bm d(x,x')<D$ is easier, so we will treat it first. By Lemma \ref{lem: F concentration no cutlocus}, we have in this case $$G^{\epsilon'}_{t,x,x'}(at,y)\leq G^{\epsilon'}_{a(1-a)t/\zeta}(y,\gamma_{xx'}(a)).$$
    This gives us \begin{align*}
        \int_M G^{\epsilon'}_{t,x,x'}(at,y)m(dy)\leq&\int_M G^{\epsilon'}_{a(1-a)t/\zeta}(y,\gamma_{xx'}(a))m(dy) \\
       \text{and}\quad \int_M G^{\epsilon'}_{t,x,x'}(at,y)\bm d(y,x_0)^{2\alpha-d}m(dy)\leq& \int_M G^{\epsilon'}_{a(1-a)t/\zeta}(y,\gamma_{xx'}(a))\bm d(y,x_0)^{2\alpha-d}m(dy).
    \end{align*}
    Since $\gamma_{xx'}(a)$ is simply a point in $M$, \eqref{est: bridge density} follows from from the first right hand side above and \eqref{est:1pt Gauss est}, and \eqref{est: bridge density with riesz} likewise follows from the second right hand side above and \eqref{est: pt gauss riesz 1 int}.
    \\
    In the case $\bm d(x,x')\geq D$, we have by Lemma \ref{lem: F concentration}$$G^{\epsilon'}_{t,x,x'}(s,y)=t^{\frac{d-1}{2}}(a(1-a)t)^{-\frac{d}{2}}\exp\left(-\frac{F_{a;x,x'}(y)}{c_{\epsilon'}a(1-a)t}\right)\leq t^{\frac{d-1}{2}}(a(1-a)t)^{-\frac{d}{2}}\exp\left(-\frac{|\bm d(x,y)-a\bm d(x,x')|^2}{c_{\epsilon'}a(1-a)t}\right).$$ 
    Before proceeding further, we will introduce some useful notation. Let $\exp_x$ denote the exponential map based at $x\in M$. For $x\in M$ denote by $\overline{M_x}$ the subset of $T_xM$ determined as the largest open subset for which the corresponding geodesic rays are unique, non-conjugate minimizers. Equivalently, it is given by all geodesic rays prior to their cut times. Unfortunately, there is no broad agreement on the name for such a fundamental object, but \cite[Chapter 10]{leeRmMfd}, which we take as our basic reference for what follows, calls it the injectivity domain of $x$. In particular, it is shown there that $\overline{M_x}$ is well defined, and \cite[Theorem 10.34]{leeRmMfd} gives the basic propeties
    \begin{enumerate}
        \item $0\in \overline{M_x}$. (Recall that $\exp_x(0)=x$).
        \item $\exp_x:\overline{M_x}\to M$ is a diffeomorphism onto its image.
    \end{enumerate}
    Then the facts \begin{enumerate}
        \item $\exp_x(\overline{M_x})=M\setminus \mathrm{Cut}_x$. Note that this implies $\int_M m(dy)=\int_{\overline{M_x}}\exp^*_xm(dy)$, where $\exp^*_xm$ denotes the pullback of $m$ by $\exp_x$.
        \item $\overline{M_x}\subset B^{T_xM}(0,D_M)$, where $B^{T_xM}(0,r)$ is the Euclidean ball of radius $r$ on $T_xM\cong \R^d$ centered at $0=\exp^{-1}_x(M)\cap\overline{M_x}$.
        \item As measures, $\exp^*m(dy)|_{\overline{M_x}}\leq C dy$, where $dy$ is the Lebesgue measure on $T_xM\cong \R^d$.
    \end{enumerate}
    follow from this along with the metric comparison result \cite[Theorem 11.10]{leeRmMfd} and the fact that $M$ is compact, so that all sectional curvatures are bounded from below and $\overline{M_x}$ is bounded. The above three facts along with Gauss' lemma imply that for \eqref{est: bridge density} we have \begin{align*}
        \int_M G^{\epsilon'}_{t,x,x'}(at,y)m(dy)\leq& t^{\frac{d-1}{2}}\int_M (a(1-a)t)^{-\frac{d}{2}}\exp\left(-\frac{|\bm d(x,y)-a\bm d(x,x')|^2}{c_{\epsilon'}a(1-a)t}\right) m(dy)\\
        =& t^{\frac{d-1}{2}}\int_{\overline{M_x}} (a(1-a)t)^{-\frac{d}{2}}\exp\left(-\frac{{||\exp_{x}^{-1}(y)||}_{T_x(M)}-a\bm d(x,x')|^2}{c_{\epsilon'}a(1-a)t}\right) \exp_x^*m(dy)\\
        \leq& t^{\frac{d-1}{2}} \int_{B^{\R^d}(0,D_M)} (a(1-a)t)^{-\frac{d}{2}}\exp\left(-\frac{|{|y|_{\R^d}}-a\bm d(x,x')|^2}{c_{\epsilon'}a(1-a)t}\right)  dy\\
        =&t^{\frac{d-1}{2}} \int_{\S^{d-1}}\int_0^{D_M} (a(1-a)t)^{-\frac{d}{2}}\exp\left(-\frac{|r-a\bm d(x,x')|^2}{c_{\epsilon'}a(1-a)t}\right)r^{d-1} drd\theta\\
        =& Ct^{\frac{d-1}{2}} \int_0^{D_M} (a(1-a)t)^{-\frac{d}{2}}\exp\left(-\frac{|r-a\bm d(x,x')|^2}{c_{\epsilon'}a(1-a)t}\right)r^{d-1} dr\\
        =&Ct^{\frac{d-1}{2}} \left(\int_0^{a\bm d(x,x')}+\int_{a\bm d(x,x')}^{D_M}\right) (a(1-a)t)^{-\frac{d}{2}}\exp\left(-\frac{|r-a\bm d(x,x')|^2}{c_{\epsilon'}a(1-a)t}\right)r^{d-1} dr.
    \end{align*}
    Denote the last two integrals above by $\int_0^{a\bm d(x,x')}=:I_<$ and $\int_{a\bm d(x,x')}^{D_M}=:I_>$. We can conclude the proof of \eqref{est: bridge density} if we can show that $t^{\frac{d-1}{2}}I_<$ and $ t^{\frac{d-1}{2}}I_>$ are both less than or equal to $C_{\epsilon,M}$. We will show $t^{\frac{d-1}{2}}I_>\leq C$, the argument for $I_<$ is nearly identical. Note that on $[a\bm d(x,x'),D_M]$, we have $|r-a\bm d(x,x')|=r-a\bm d(x,x')$, so the change of variable $R=\frac{r-a\bm d(x,x')}{(c_{\epsilon'}a(1-a)t)^{\frac{1}{2}}}$ gives us 
    \begin{align*}
        t^{\frac{d-1}{2}}I_>\leq&t^{\frac{d-1}{2}}(c_{\epsilon'})^{\frac{d-1}{2}}{\int_0^{+\infty}} \exp(-R^2) \left(R+\sqrt{\frac{a}{c_{\epsilon'}(1-a)t}}\bm d(x,x')\right)^{d-1} dR\\
        \leq& t^{\frac{d-1}{2}}(c_{\epsilon'})^{\frac{d-1}{2}}\int_0^{+\infty} \exp(-R^2) \left(R+\sqrt{\frac{1}{c_{\epsilon'}t}}\bm d(x,x')\right)^{d-1} dR\\
        \leq& (t^{\frac{d-1}{2}}\wedge1) (c_{\epsilon'})^{\frac{d-1}{2}}\leq (c_{\epsilon})^{\frac{d-1}{2}} = C_{\epsilon,M}.
    \end{align*}
    Finally, we will prove \eqref{est: bridge density with riesz} by combining the strategies used for proving \eqref{est: pt gauss riesz 1 int} and \eqref{est: bridge density}. We can decompose $M$ as $M=M_{x,x'}\cup M_{x_0}$, where
    \[\begin{split}
    M_{x,x'}:&=\set{y\in M:|\bm d(x,y)-a\bm d(x,x'))|\leq \bm d(x_0,y)} \\ \text{and}\quad M_{x_0}:&=\set{y\in M:|\bm d(x,y)-a\bm d(x,x'))|\geq \bm d(x_0,y)}.
    \end{split}\]
    As with $M_x,M_{x'}$ used in the proof of \eqref{est: pt gauss riesz 1 int}, the closedness (and therefore $m-$measurability) of $M_{x,x'}$ and $M_{x_0}$ follows from the continuity of $f_{x,x',x_0}(y):=|\bm d(x,y)-a\bm d(x,x'))|-\bm d(x_0,y)$. Applying Lemma \ref{lem: F concentration}, we have \begin{align*}
        &\int_M G^{\epsilon'}_{t,x,x'}(at,y)\bm d(y,x_0)^{2\alpha-d}m(dy)\\
        \leq&t^{\frac{d-1}{2}}\int_M (a(1-a)t)^{-\frac{d}{2}}\exp\left(-\frac{|\bm d(x,y)-a\bm d(x,x')|^2}{c_{\epsilon'}a(1-a)t}\right) \bm d(y,x_0)^{2\alpha-d}m(dy)\\
        \leq &t^{\frac{d-1}{2}}\left(\int_{M_{x,x'}}+\int_{M_{x_0}}\right) (a(1-a)t)^{-\frac{d}{2}}\exp\left(-\frac{|\bm d(x,y)-a\bm d(x,x')|^2}{c_{\epsilon'}a(1-a)t}\right)\bm d(y,x_0)^{2\alpha-d}m(dy)\\
        \leq& t^{\frac{d-1}{2}} \int_{M_{x,x'}} (a(1-a)t)^{-\frac{d}{2}}\exp\left(-\frac{|\bm d(x,y)-a\bm d(x,x')|^2}{c_{\epsilon'}a(1-a)t}\right) |\bm d(x,y)-a\bm d(x,x')|^{2\alpha-d} m(dy)\\
        &+t^{\frac{d-1}{2}}\int_{M_{x_0}} (a(1-a)t)^{-\frac{d}{2}}\exp\left(-\frac{\bm d(y,x_0)^2}{c_{\epsilon'}a(1-a)t}\right) \bm d(y,x_0)^{2\alpha-d}m(dy)=I_{x,x'}+I_{x_0}.
    \end{align*}
    Since $t^{\frac{d-1}{2}}$ is bounded for $t\in (0,1)$, $I_{x_0}\leq C_{\epsilon,M}[(a(1-a)t)^{\frac{2\alpha-d}{2}}+1]$ follows from applying \eqref{est: pt gauss riesz 1 int}. We now only need to show \[\label{est:I_xx' bd}I_{x,x'}\leq C_{\epsilon,M}[(a(1-a)t)^{\frac{2\alpha-d}{2}}+1].\]
    Taking normal coordinates at $x$, we have \begin{align*}
        I_{x,x'}=& t^{\frac{d-1}{2}}\int_{\exp_x^{-1}(M_{x,x'})} (a(1-a)t)^{-\frac{d}{2}}\exp\left(-\frac{||\exp_{x}^{-1}(y)|_{T_x(M)}-a\bm d(x,x')|^2}{c_{\epsilon'}a(1-a)t}\right) \\
         &\times||\exp_{x}^{-1}(y)|_{T_x(M)}-ad(x,x')|^{2\alpha-d}\exp_x^{*}m(dy)\\
        \leq& C t^{\frac{d-1}{2}}\int_{B^{\R^d}(D_M)} (a(1-a)t)^{-\frac{d}{2}}\exp\left(-\frac{||y|_{\R^d}-a\bm d(x,x')|^2}{c_{\epsilon'}a(1-a)t}\right) ||y|_{\R^d}-a\bm d(x,x')|^{2\alpha-d} dy\\
        =&C t^{\frac{d-1}{2}}\int_0^{D_M} (a(1-a)t)^{-\frac{d}{2}}\exp\left(-\frac{|r-a\bm d(x,x')|^2}{c_{\epsilon'}a(1-a)t}\right) |r-a\bm d(x,x')|^{2\alpha-d}r^{d-1} dr.
    \end{align*}
    From here we can repeat what was done for \eqref{est: bridge density}, with the extra $(a(1-a)t)^{\frac{2\alpha-d}{2}}$ term in the right hand side of the desired estimate emerging from the extra $|r-a\bm d(x,x')|^{2\alpha-d}$ factor in the integrand. This concludes the proof.
\end{proof}
{
\begin{remark}
    Observe in the proof of \eqref{est: bridge density} that the $t^{\frac{d-1}{2}}$ factor in front of the integral $I_>$, which came from using the short time upper bound in Lemma \ref{lem: heat kernel bd} in the definition of $G^{\varepsilon}_{t,x,y}(s,z)$, is necessary to make the entire quantity bounded. In the case $M=\mathbb{S}^d$ with $x,x'$ being antipodal points, the estimate which invoked Lemma \ref{lem: F concentration} becomes an equality. Something similar, which we omitted writing to avoid repetition, happens for the proof of \eqref{est: bridge density with riesz}, which ultimately is what will allow us to conclude that the sharp Dalang condition $\alpha>\frac{d-2}{2}$ holds.
    Since Lemma \ref{lem: F concentration no cutlocus} is proven by comparing to a sphere with the proper sectional curvature, one can interpret the results in Section \ref{sec:geod and F} as the geometric structures which converted a would-be sphere-only proof into an all-compact-Riemannian-manifolds proof of Theorem \ref{thm: main result}.
\end{remark}
}
Lemma \ref{lem: single integral bounds} implies the following double space integral estimates.
\begin{lemma}\label{lem: double integral bounds}
    Let the assumptions in Lemma \ref{lem: single integral bounds} prevail. The following estimates hold with the constants on the right hand side possibly depending on $\epsilon$ if the left hand side has dependence on $\epsilon'$.
    \begin{equation}\label{est:2int 1LY Riesz}
    {\sup_{x\in M}}\iint_{M^2} G^{\epsilon'}_t(\bm d(x,y))\bm d(y,y')^{2\alpha-d} m(dy)m(dy')\leq C,\end{equation}
    \begin{equation}\label{est:2int 2LY Riesz}
    {\sup_{x,x'\in M}}\iint_{M^2} G^{\epsilon'}_t(\bm d(x,y))\bm d(y,y')^{2\alpha-d}G^{\epsilon'}_t(\bm d(x',y)) m(dy)m(dy')\leq C(t^{\frac{2\alpha-d}{2}}+1).\end{equation}
    Furthermore, {we have for all $0<t<1$ and  $a\in(0,\frac{1}{2})$,}
    \begin{equation}\label{est:2int 1Bridge Riesz}{\sup_{x,z\in M}}\iint_{M^2}G^{\epsilon'}_{t,x,z}(at,y)\bm d(y,y')^{2\alpha-d}  m(dy)m(dy')\leq C,\end{equation}
    \begin{equation}\label{est:2int 1Bridge 1LY Riesz}
    {\sup_{x,x',z\in M}}\iint_{M^2} G^{\epsilon'}_{t,x,z}(at,y)\bm d(y,y')^{2\alpha-d}G^{\epsilon'}_{a(1-a)t}(\bm d(x',y')) m(dy)m(dy')\leq C[(a(1-a)t)^{\frac{2\alpha-d}{2}}+1],\end{equation}
    \begin{equation}\label{est:2int 2Bridge Riesz}
    {\sup_{x,x',z,z'\in M}}\iint_{M^2} G^{\epsilon'}_{t,x,z}(at,y)\bm d(y,y')^{2\alpha-d}G^{\epsilon'}_{t,x',z'}(at,y) m(dy)m(dy')\leq C[(a(1-a)t)^{\frac{2\alpha-d}{2}}+1].\end{equation}
\end{lemma}
\begin{proof}
    These all follow from applying the $1-\infty$ H\"older inequality in one space integral then applying the appropriate estimates from Lemma \ref{lem: single integral bounds} to what remains. We will work this out for \eqref{est:2int 1Bridge 1LY Riesz}, from which it will be obvious that the same argument works for all others.\\
    Define the function $$f_{x',a,t}(y):=\int_M \bm d(y,y')^{2\alpha-d}G^{\epsilon'}_{a(1-a)t}(\bm d(x',y')) m(dy').$$
    Now we can rewrite the left hand side of \eqref{est:2int 1Bridge 1LY Riesz} and apply $1-\infty$ H\"older: \begin{align*}
        &\iint_{M^2} G^{\epsilon'}_{t,x,z}(at,y)\bm d(y,y')^{2\alpha-d}G^{\epsilon'}_{a(1-a)t}(d(x',y')) m(dy)m(dy')\\
        =&\int_M G^{\epsilon'}_{t,x,z}(at,y)f_{x',a,t}(y) m(dy)
        \leq \int_M G^{\epsilon'}_{t,x,z}(at,y)m(dy)\left(\sup_{x'\in M}\sup_{y\in M}f_{x',a,t}(y)\right)
    \end{align*}
    Now apply \eqref{est: bridge density} and \eqref{est: pt gauss riesz 1 int} to the single integrals to conclude the proof.
\end{proof}

\bigskip

In the next three subsections, we focus on establishing the desired upper bounds for $\mathcal{L}_n$. This is accomplished by first proving the corresponding estimate for $\mathcal{L}_1$, and then extending it to general $\mathcal{L}_n$ via an iteration argument. To set up this procedure, we first introduce some notation. Recall the definitions of $G^\epsilon_t$ and $\tilde{G}^\epsilon_t$ in
\eqref{eq: LY and Hsu Gaussian def}, as well as that of $G^\epsilon_{t,x,y}(s,z)$ in \eqref{def: quotients of different UB}. Throughout the remainder of the paper, we fix an $\epsilon>0$ once and for all, and for each $n\geq 1$ define
\[
\epsilon_n := \Bigl(1-\frac{1}{n}\Bigr)\epsilon, \qquad
c_n := c_{\epsilon_n},
\]
together with 
\begin{align}\label{def: G^ns}
{G^{[n]}_t(r) := G^{\epsilon_n}_t(r), \qquad
\tilde{G}^{[n]}_t(r) := \tilde{G}^{\epsilon_n}_t(r), \qquad
G^n_{t,x,y}(s,z) := G^{\epsilon_n}_{t,x,y}(s,z).}
\end{align}
By construction, $\epsilon_1<\epsilon_n<\epsilon$ for all $n$, and consequently
$c_1<c_n<c_\epsilon$.

\subsection{Proof of Theorem \ref{thm: L1 bd} for $t\ge1$}
{With the above preparations in place, we are now ready to prove the estimate
claimed in Theorem~\ref{thm: L1 bd}. We begin by considering the case where the time parameter $t$ is large, for which the following result holds.}
\begin{lemma}\label{lem: large time L1 bd}
    Suppose $W$ satisfies Dalang's condition $\alpha>{(d-2)}/{2}$ and $t\ge1$. Define a function $$k^{1}_L(s):=\sup_{x,x'\in M}\iint_{M^2}  [G^{[1]}_s(\bm d(x,z))+1][G^{[1]}_s (\bm{d}(x',z'))+1]\bm{d}(z,z')^{2\alpha-d}m(dz)m(dz'),\quad s>0.$$
    We have $$k^1_L(s)\leq C_{\alpha,M}(1+s^{\frac{2\alpha-d}{2}}),$$
    for some positive constant $C_{\alpha,M}$ depending on $\alpha$ and $M$. Moreover,
    \begin{align}\label{bound:L1 by k1}\mathcal{L}_1(t,x_0,x,x_0',x')\leq C_L [G^{[1]}_t(\bm d(x_0,x))+1][G^{[1]}_t(\bm d(x_0',x'))+1]\left(\int_0^t k^1_L(s)ds\right),\end{align}
    where $C_L$ is a positive constant depending on $M$.
\end{lemma}

{The analysis in the large-time regime is considerably simpler than in the short-time case, as it does not require a detailed understanding of the geometry of geodesics. Although the argument in \cite{CO25} can be adapted \textit{mutatis mutandis} to the present setting, we provide a brief alternative proof here for the sake of completeness and to fit our current formulation.}

\begin{proof}
Recall the definition of $\mathcal{L}_n$ in \eqref{def: L_n}, we have 
\begin{align}\label{eq: L_1}\mathcal{L}_1(t,x_0,x,x_0',x')=\int_0^t ds\int_{M^2}dzdz' P_{t-s}(x_0,z)P_s(z,x)P_{t-s}(x_0',z')P_s(z',x')\bm{G}_{\alpha,\rho}(z,z').\end{align}
    Use the symmetry of $s$ and $t-s$ to reduce the time integral to that over $\int_0^{t/2}$, and observe that for $t\ge 1$, $$\frac{G^{[1]}_{t-s}(\bm d(z,x))+1}{G^{[1]}_{t}(\bm d(x_0,x))+1}\leq C.$$
    Thus, applying Corollary \ref{Prop:LiYau} to the heat kernels in the right-hand side of \eqref{eq: L_1} with gives us \eqref{bound:L1 by k1}.
    This leaves us to show the estimate for $k^1_L(s)$. For that, we simply write \begin{align*}
        &\iint_{M^2} [G^{[1]}_s(x,z)+1][G^{[1]}_s(x',z')+1]\bm{d}(z,z')^{2\alpha-d} m(dz)m(dz')\\
        =&\iint_{M^2} G^{[1]}_s(x,z)G^{[1]}_s(x',z') \bm{d}(z,z')^{2\alpha-d} m(dz)m(dz')+\iint_{M^2} G^{[1]}_s(x,z) \bm{d}(z,z')^{2\alpha-d} m(dz)m(dz')\\
        &+\iint_{M^2} G^{[1]}_s(x',z') \bm{d}(z,z')^{2\alpha-d} m(dz)m(dz')+\iint_{M^2}  \bm{d}(z,z')^{2\alpha-d}1 m(dz)m(dz')=:I_{xx'}+I_{x}+I_{x'}+I.
    \end{align*}
    The proof is concluded by applying
    \eqref{est: Riesz} to $I$, \eqref{est:2int 1LY Riesz} to $I_x$ and $I_{x'}$, and \eqref{est:2int 2LY Riesz} to $I_{xx'}$.
\end{proof}

\subsection{Proof of Theorem \ref{thm: L1 bd} for $0<t<1$}

We now turn to the small-time regime, where the careful treatment of geodesic geometry done in Section \ref{sec:geod and F} (in particular, applications of \eqref{est: bridge density} and \eqref{est: bridge density with riesz}) becomes essential. 
In order to state the main result of this section, we first introduce some notation. Recall the definitions of
$G^{[n]}_t$ and $\tilde{G}^{[n]}_t$ in \eqref{def: G^ns}. We begin with the following
decomposition:
\begin{align}\label{decomp of upper bounds}
   &\frac{\bigl[G^{[n]}_{t-s}(\bm{d}(x_0,z))+1\bigr]\bigl[G^{[n]}_s(\bm{d}(z,x))+1\bigr]}
   {\1_{\bm d(x_0,x)< D}G^{[n]}_t(\bm d(x_0,x))+\1_{\bm d(x_0,x)\geq D}\tilde{G}^{[n]}_t(\bm d(x_0,x))+1}\nonumber \\
   \leq\;&
   \frac{G^{[n]}_{t-s}(\bm{d}(x_0,z)) G^{[n]}_{s}(\bm{d}(z,x))}
   {\1_{\bm d(x_0,x)< D}G^{[n]}_t(\bm d(x_0,x))+\1_{\bm d(x_0,x)\geq D}\tilde{G}^{[n]}_t(\bm d(x_0,x))} \nonumber\\
   &\quad
   + \frac{\text{terms containing at most one Gaussian factor}}{C} \nonumber\\
   \leq\;&
   G^{n}_{t,x_0,x}(s,z)
   + \bigl(G^{[n]}_{t-s}(\bm d(x_0,z))+G^{[n]}_s(\bm d(z,x))\bigr)
   + C_H \nonumber\\
   =:\;&
   G^n_{t,x_0,x}(s,z)+f^n_{t,x_0,x}(s,z)+C_H .
\end{align}

To simplify notation, whenever no confusion may arise we write $G^n(*)$ and
$G^n(*')$ (respectively, $f^n(*)$ and $f^n(*')$) to denote
$G^n_{t,x,y}(s,z)$ (respectively, $f^n_{t,x,y}(s,z)$), with the understanding
that the variables may be $(x,y,z)$ or $(x',y',z')$. With this convention, we
obtain
\begin{align}\label{quotient bound by G f}
    &\frac{\bigl[G^{[n]}_{t-s}(\bm{d}(x_0,z))+1\bigr]\bigl[G^{[n]}_{t-s}(\bm{d}(z,x))+1\bigr]}
    {\1_{\bm d(x_0,x)< D}G^{[n]}_t(\bm d(x_0,x))+\1_{\bm d(x_0,x)\geq D}\tilde{G}^{[n]}_t(\bm d(x_0,x))+1} \nonumber\\
    &\quad\times
    \frac{\bigl[G^{[n]}_{t-s}(\bm{d}(x_0',z'))+1\bigr]\bigl[G^{[n]}_s(\bm{d}(z',x'))+1\bigr]}
    {\1_{\bm d(x_0',x')< D}G^{[n]}_t(\bm d(x_0',x'))+\1_{\bm d(x_0',x')\geq D}\tilde{G}^{[n]}_t(\bm d(x_0',x'))+1} \nonumber\\
    \leq\;&
    G^n(*)G^n(*')
    + f^n(*)f^n(*')
    + C_H^2 \nonumber\\
    &\quad
    + G^n(*)f^n(*')
    + G^n(*')f^n(*)
    + C_H\bigl[f^n(*)+f^n(*')+G^n(*)+G^n(*')\bigr].
\end{align}
We denote the right-hand side of \eqref{quotient bound by G f} by
$R^n_{t,x_0,x,x_0',x'}(s,z,z')$.

\begin{lemma}\label{lem:short time L1 bound}
    Assume ${d}/{2}>\alpha>{(d-2)}/{2}$.
    Define for each $s>0$ ,$$k^1_S(s):=\sup_{t\geq 2s}\sup_{x_0,x,x_0',x'\in M}\iint_{M^2}R^1_{t,x_0,x,x_0',x'}(s,z,z')\bm{d}(z,z')^{2\alpha-d}m(dz)m(dz').$$
    We have $$k^1_S(s)\leq C_M(1+s^{\frac{2\alpha-d}{2}}),\quad\mathrm{for\ all}\ s>0,$$
    for some positive constant depending on $M$. In addition, for all $t>0$
    $$\mathcal{L}_1(t,x_0,x,x_0',x')\leq C_S [G^{[2]}_t(\bm d(x_0,x))+1][G^{[2]}_t(\bm d(x_0',x'))+1]\left(\int_0^t k^1_S(s)ds\right),$$
    where $C_S$ depends on $\alpha$ and $M$.
    \end{lemma}
\begin{proof}
    As in the proof for Lemma \ref{lem: large time L1 bd}, we reduce the time integral to the same integral over $\int_0^{t/2}$ so that we have $0<s<\frac{t}{2}$ and so $a:=\frac{s}{t}\in(0,\frac{1}{2})$. We next apply Corollary \ref{Prop:LiYau} to all heat kernels in the integrand of $\mathcal{L}_1$ with $\epsilon=\epsilon_1$, then multiply and divide by \begin{align*}
        \mathcal{G}^1(t,x,x_0,x',x_0'):=&[\1_{\bm d(x_0,x)< D}G^{[1]}_t(\bm d(x_0,x))+\1_{\bm d(x_0,x)\geq D}\tilde{G}^{[1]}_t(\bm d(x_0,x))+1]\\
        \times&[\1_{\bm d(x_0',x')< D}G^{[1]}_t(\bm d(x_0',x'))+\1_{\bm d(x_0',x')\geq D}\tilde{G}^{[1]}_t(\bm d(x_0',x'))+1].
    \end{align*}
    By \eqref{est: LY beats LY larger eps} and \eqref{est: LY beat Hsu larger eps}, we have$$\mathcal{G}^1(t,x,x_0,x',x_0')\leq C [G^{[2]}_t(\bm d(x_0,x))+1][G^{[2]}_t(\bm d(x_0',x'))+1].$$ This gives us \begin{align*}
        & \mathcal{L}_1(t,x,x_0,x,x_0',x')\\
        \leq& C\int_0^t \iint_{M^2}  \bm d(z,z')^{2\alpha-d}[G^{[1]}_{t-s}(\bm d(x_0,z))+1][G^{[1]}_{s}(\bm d(z,x))+1]\\
        &\times [G^{[1]}_{t-s}(\bm d(x_0',z'))+1][G^{[1]}_{s}(\bm d(z',x'))+1]m(dz)m(dz')ds\\
        =&C \mathcal{G}^1(t,x,x_0,x',x_0')\int_0^{\frac{t}{2}}\iint_{M^2} R^1_{t,x_0,x,x_0',x'}(s,z,z') \bm{d}(z,z')^{2\alpha-d}m(dz)m(dz')ds\\
        \leq& C [G^{[2]}_t(\bm d(x_0,x))+1][G^{[2]}_t(\bm d(x_0',x'))+1]\int_0^{\frac{t}{2}} \iint_{M^2}R^1_{t,x_0,x,x_0',x'}(s,z,z')\bm{d}(z,z')^{2\alpha-d}m(dz)m(dz')ds\\
    \end{align*}
    We have obtained the desired $\mathcal{L}_1$ bound, so it remains to prove the estimate for $k^1_S(s)$. For this, one only needs to decompose the integral into finitely many integrals using \eqref{quotient bound by G f}, then apply Lemma \ref{lem: double integral bounds} in the following way.
    \begin{enumerate}
        \item For the constant term, apply \eqref{est: Riesz}.
        \item For the $Cf^1$ terms, apply \eqref{est:2int 1LY Riesz}.
        \item For the $CG^1$ terms, apply \eqref{est:2int 1Bridge Riesz}.
        \item For the $f^1f^1$ term, apply \eqref{est:2int 2LY Riesz}.
        \item For the $f^1G^1$ terms, apply \eqref{est:2int 1Bridge 1LY Riesz}.
        \item For the $G^1G^1$ term, apply \eqref{est:2int 2Bridge Riesz}.
    \end{enumerate}
    Items 1--3 provide constant bounds for the integrals of the corresponding terms, while items 4--6 provide \(C(s^{\frac{2\alpha-d}{2}}+1)\) bounds instead. For example, \begin{align*}
        &\int_0^{\frac{t}{2}}\iint_{M^2}G^{1}(*)\bm{d}(z,z')^{2\alpha-d}m(dz)m(dz')ds\\
        =& \int_0^{\frac{t}{2}}\iint_{M^2} G^1_{t,x_0,x}(s,z)\bm{d}(z,z')^{2\alpha-d}m(dz)m(dz')ds\leq C_{\alpha,M}
    \end{align*}
    by \eqref{est:2int 1Bridge Riesz}, and \begin{align*}
        &\int_0^{\frac{t}{2}}\iint_{M^2}G^{1}(*)G^{1}(*')\bm{d}(z,z')^{2\alpha-d}m(dz)m(dz')ds\\
        =& \int_0^{\frac{t}{2}}\iint_{M^2} G^1_{t,x_0,x}(s,z)G^1_{t,x_0',x'}(s,z') \bm{d}(z,z')^{2\alpha-d}m(dz)m(dz')ds \leq C_{\alpha,M}(s^{\frac{2\alpha-d}{2}}+1)
    \end{align*}
    by \eqref{est:2int 2Bridge Riesz}. Collecting these estimates completes the proof.
\end{proof}

\subsection{Bounding $\mathcal{L}_n$}\label{Sec Ln bd}

{We now carry out the induction used to control $\mathcal{L}_n$, modified to account for the geometric decomposition described above. As noted earlier, the argument requires a slight modification from previous versions of the induction in the literature because two different heat kernel upper bounds were applied, and it relies on the elementary relations \eqref{est: LY beats LY larger eps}--\eqref{est: LY beat Hsu larger eps}, which allows us to kill off the unwanted $t^{-\frac{d-1}{2}}$ factor in $\Tilde{G}^{[n]}$ (which would prevent the next induction step from being carried out since it makes the integrand no longer small time integrable if left untreated) by absorbing it into $G^{[n+1]}$ at the end of the induction.} Our main result is the following Lemma.
\begin{theorem}\label{thm: Ln bd}
    Let $\alpha\in \left(\frac{d-2}{2},\frac{d}{2}\right)$ and $C$ be the constant from Theorem \ref{thm: L1 bd}. Recall the definition of $R^n$ from \eqref{quotient bound by G f}, and define a sequence of functions $k^n$ on $\R_+$ by \begin{align*}
        k^n(s):=&k^n_L(s)+k^n_S(s),\quad\text{where}\\
        k^n_L(s):=&\sup_{x,x'\in M}\iint_{M^2}  [G^{[n]}_s(x,z)+1][G^{[n]}_s(x',z')+1]\bm{d}(z,z')^{2\alpha-d}m(dz)m(dz'),\\
        k^n_S(s):=&\sup_{t\geq 2s}\sup_{x_0,x,x_0',x'\in M}\iint_{M^2}R^n_{t,x_0,x,x_0',x'}(s,z,z')\bm{d}(z,z')^{2\alpha-d}m(dz)m(dz').
    \end{align*}
    For a constant $C_M$ independent of $n$, these satisfy
    \[k^n(s)\leq C_M(1+s^{\frac{2\alpha-d}{2}}).\]
    Then, define iteratively another sequence of functions $h_n$ on $\R_+$ by \[\label{def:hn} h_0(t)=1,\quad h_{n+1}(t):=\int_0^t h_n(t-s)k^n(s)ds.\]
    For all $n\geq 0$, $h_n$ is non-negative and non-decreasing. Furthermore, for all $n\geq 1$ and $x_0,x,x_0',x'\in M$, we have \[\label{est: Ln} \mathcal{L}_n(t,x_0,x,x_0',x')\leq (2C)^n[G^{[[n+1]}_t(d(x_0,x))+1][G^{[n+1]}_t(d(x_0',x'))+1]h_n(t).\]
\end{theorem}

\begin{proof}
    The bound for $k^n$ follows from showing it for $k^n_L$ and $k^n_S$, which requires no changes from the bound for $k^1_L$ and $k^1_S$ since all $\epsilon_n<\epsilon$.
    \medskip
    
    The non-negativity and non-decreasing property of all $h_n$ can be proven by induction using a similar argument to that of \cite[Lemma 3.21]{CO25}, where the only change necessary is that in the induction step assuming it holds up to $n$, the non-negativity of $k^n$ is used instead of one singular function $k$ for all $n$.
    \medskip
    
    For the $\mathcal{L}_n$ bound, the case $n=1$ is part of the statement of Theorem \ref{thm: L1 bd}. We now suppose the result holds up to $n$, and will prove it for $n+1$. By induction hypothesis, we have \begin{align*}
        &\mathcal{L}_{n+1}(t,x_0,x,x_0',x')\\
        =&\int_0^t \iint_{M^2}\mathcal{L}_0(t-s,x_0,y,x_0',y')\bm{G}_{\alpha,\rho}(y,y')\mathcal{L}_n(s,y,x,y',x')m(dy)m(dy')ds\\
        \leq& (2C)^n\int_0^t \iint_{M^2}\mathcal{L}_0(t-s,x_0,y,x_0',y')|\bm{G}_{\alpha,\rho}(y,y')|\\
        &\times [G^{[n+1]}_s(\bm d(y,x))+1][G^{[n+1]}_s(\bm d(y',x'))+1]m(dy)m(dy')h_n(s)ds\\
        \leq& (2C)^n\int_0^t\iint_{M^2}[G^{[n+1]}_{t-s}(\bm d(x_0,y))+1][G^{[n+1]}_{t-s}(\bm d(x_0',y'))+1]|\bm{G}_{\alpha,\rho}(y,y')|\\
        &\times [G^{[n+1]}_s(\bm d(y,x))+1][G^{[n+1]}_s(\bm d(y',x'))+1]m(dy)m(dy')h_n(s)ds
    \end{align*}
    By the non-negativity of all $h_n$, the symmetry of $t-s$ and $s$ in the time integrand above, and the symmetry of the roles of $x_0,x$ and $x_0',x'$ respectively, one can reduce the time integral above to an integral over $\int_{t/2}^tds$, from which a change of variables $t-s\mapsto s$ gives \begin{align*}
        & \int_0^t\iint_{M^2}[G^{[n+1]}_{t-s}(\bm d(x_0,y))+1][G^{[n+1]}_{t-s}(\bm d(x_0',y'))+1]|\bm{G}_{\alpha,\rho}(y,y')|\\
        &\times [G^{[n+1]}_s(\bm d(y,x))+1][G^{[n+1]}_s(\bm d(y',x'))+1]m(dy)m(dy')h_n(s)ds\\
        \leq&2\int_0^{\frac{t}{2}}\iint_{M^2}[G^{[n+1]}_{s}(\bm d(x_0,y))+1][G^{[n+1]}_{s}(\bm d(x_0',y'))+1]|\bm{G}_{\alpha,\rho}(y,y')|\\
        &\times [G^{[n+1]}_{t-s}(\bm d(y,x))+1][G^{[n+1]}_{t-s}(\bm d(y',x'))+1]m(dy)m(dy')h_n(t-s)ds.
    \end{align*}
    If $t\ge 1$, the last $\iint_{M^2}$ above can be treated as in the proof of Lemma \ref{lem: large time L1 bd} by dividing and multiplying by $[G^{[n+1]}_t(\bm d(x_0,x))+1][G^{[n+1]}_t(\bm d(x_0',x'))+1]$. For $0<t<1$, divide and multiply by 
    \begin{align*}
        &[\1_{d(x_0,x)<D}G^{[n]}_t(\bm d(x_0,x))+\1_{\bm d(x_0,x)\geq D}\tilde{G}^{[n]}_t(\bm d(x_0,x))+1]\\
        \times&[\1_{\bm d(x_0',x')< D}G^{[n]}_t(\bm d(x_0',x'))+\1_{\bm d(x_0',x')\geq D}\tilde{G}^{[n]}_t(\bm d(x_0',x'))+1],
    \end{align*} instead and the last $\iint_{M^2}$ above can be treated as in the proof of Lemma \ref{lem:short time L1 bound}. The proof is finished after appropriately applying \eqref{est: LY beats LY larger eps} and \eqref{est: LY beat Hsu larger eps}.
\end{proof}

\subsection{Well-Posedness and Moment Upper Bound}\label{Sec: existence and moment bound}
We are now ready to prove the well-posedness and moments upper bounds for equation \eqref{eq: SHE}. Recall the iteration procedure outlined in Section 2. In particular, equation \eqref{iteration corelation} implies that the existence of an $L^2$-solution to \eqref{eq: SHE} relies on the convergence of the series,
\begin{equation*}
    \mathcal{K}_\beta(t,x,z,x',z')=\sum_{n=0}^\infty \beta^{2n}\mathcal{L}_n(t,x,z,x',z').
\end{equation*}
Now that $\mathcal{L}_n$ is controlled by $h_n$ thanks to Theorem \ref{thm: Ln bd}, we set for any $\lambda>0$,  
$$H_\lambda(t):=\sum_{n=0}^\infty \lambda^{2n}h_n(t).$$

\begin{corollary}\label{cor:K beta bound}
    For any $t>0$, $\epsilon>0$ and $x,x_0,x_0',x'\in M$, we have \begin{align}\label{boudn: K_beta}\mathcal{K}_\beta(t,x_0,x,x_0',x')\leq [G^\epsilon_t(\bm d(x_0,x))+1][G^\epsilon_t(\bm d(x_0',x'))+1]H_{2\beta^2 C}(t).\end{align}
\end{corollary}
\begin{proof}
    This follows trivially from the definition of $\mathcal{K}_\beta$, Theorem \ref{thm: Ln bd}, and $\epsilon_n<\epsilon$ for all $n\ge1$.
\end{proof}

The following result for $H_\lambda$ is needed to obtain exponential (in time) moment bounds for the solution $u$.
\begin{lemma}\label{lemma:behavior of H}
    Let $\alpha\in(\frac{d-2}{2},\frac{d}{2})$, $\lambda>0$.  There exist constants $C, \theta>0$ depending on $\alpha,\lambda$ such that for all $t>0$, $$H_{\lambda}(t)\leq C e^{\theta t}.$$
\end{lemma}
\begin{proof}
    Deviating slightly from \cite[Lemma 2.5]{ChenKim19} and \cite[Lemma 4.2]{CO25}, we observe that by the $k^n$ bound in Theorem \ref{thm: Ln bd}, we have for all $n\geq 0$, \begin{equation}\label{est:hn bd}
        h_n(t)\leq C^n\Tilde{h}_n(t),
    \end{equation}
    where $C$ is the constant from said bound, $\Tilde{h}_0(t)=h_0(t)=1$ and $\Tilde{h}_n(t):=\int_0^t \Tilde{h}_n(t-s)(1+s^{\frac{2\alpha-d}{2}})ds$. Thus, for all $\gamma>0$ we have$$\int_0^{+\infty}e^{-\gamma t}h_n(t)dt\leq C^n\int_0^{+\infty}e^{-\gamma t}\tilde{h}_n(t)dt$$
    The right hand side integral above can be treated as in the proof of \cite[Lemma 2.5]{ChenKim19} and \cite[Lemma 4.2]{CO25}, giving us our desired result.
\end{proof}

We now have all the ingredients to prove Theorem \ref{thm: main result}. Indeed, the six-step Picard iteration scheme used in \cite{ChenThesis,ChenDalangAOP} with the modifications presented in \cite{ChenKim19} is usable here to obtain $L^2(\Omega)$ continuity and the correlation formula. The same proof as Theorem~1.3 in \cite{COV23} is possible by the above estimates for the first inequality in the $p$-th moment bound. The exponential bound for the $p$-th moment is due to Lemma \ref{lemma:behavior of H}.

\section{Comparison Principles and Moment Lower Bound}
In this section, we show that the exponential growth in time of the moments of the solution $u(t,x)$, established in Theorem~\ref{thm: main result}, is sharp by proving a matching lower bound. This extends the result of \cite{CO25} from function-valued initial conditions to general measure-valued initial conditions. The main ingredients of the proof are a strong comparison principle for \eqref{eq: SHE} and the Markov property of the solution $u(t,x)$.

\subsection{H\"{o}lder regularity}

We follow the general strategy of \cite{CH19} to establish weak and strong comparison principles, which in particular imply the positivity of the solution for any positive measure initial condition. To this end, we first prove a H\"older continuity lemma. The time-increment argument in \cite{CH19} contains a gap arising from an incorrect heat-kernel estimate. The remedy lies in an additional spatial integration, which yields the bounds needed for the H\"older continuity argument. Adapting the proof to the Riemannian manifold setting also requires nontrivial technical work. We therefore present the modified argument in detail. This requires some preliminary results.\\

We will find it useful to work in normal coordinates in a ball around a given point in $M$ and to compare the Euclidean geometry induced by the identification with the tangent space with the Riemannian geometry of $M$. Recall that $i_M$ is the injectivity radius of $M$, so that for any $x\in M$ and any $R\in(0,i_M)$, the exponential map based at $x$ gives a diffeomorphism between the (open) Euclidean ball of radius $R$ around the origin and $B(x,R)\subset M$. For two points $y$ and $z$ in such a ball, we write $|y-z|$ for the Euclidean distance between $y$ and $z$. Further, for any curve $\gamma$ contained in $B(x,R)$, we let $|\gamma|_E$ be its length relative to the Euclidean metric and $|\gamma|_M$ its length relative to the metric on $M$.\\

Next, we give the basic comparison alluded to above.\\

\begin{lemma}\label{Lem:MetricComp}
For any $x\in M$, let $B=B(x,\frac{1}{3}i_M)$. Then for any (piecewise smooth) curve $\gamma$ contained in $B$, $|\gamma|_E$ and $|\gamma|_M$ are comparable, meaning that their ratio is bounded above and below by positive constants that depend only on $M$ (and not on $x$ or the particular choice of $\gamma$). Similarly, for any points $y,z\in B$, $\bm{d}(y,z)$ and $|y-z|$ are comparable.
\end{lemma}

\begin{proof}
Introduce geodesic polar coordinates $(r,\theta)\in [0,R)\times \bS^{d-1}$ on $B(x,R)$ for any $R\in(0,i_M)$. The standard Euclidean metric is $dr^2 + r^2 g_{\bS^{d-1}}d\theta^2$, where $g_{\bS^{d-1}}d\theta^2$ is the round metric on $\bS^{d-1}$, and we can write the metric on $M$ as $dr^2+r^2g_M d\theta^2$. Then because the radius of $B(x,R)$ is strictly less than the injectivity radius, it is a standard fact that $g_M$ and $g_{\bS^{d-1}}$ are comparable, in the sense that he implicit constants depend only on $M$ and $R$ (but not on $x$ by smoothness and compactness). Because we work on a fixed $M$, for convenience, we write $g_M \asymp_R g_{\bS^{d-1}}$ to denote comparability with implicit constants depending only on $R$. So we have, for any vector $v$ at any point $y\in B(x,R)$, its Euclidean length $|v|_E$ and its $M$-length $|v|_M$ satisfy $|v|_E \asymp_R |v|_M$.\\

Now if $\gamma$ is any (piecewise smooth) curve contained in $B(x,R)$, its length with respect to either metric is computed by integrating the length of its tangent vector with respect to the metric. But then it follows from the above that $|\gamma|_E \asymp_R |\gamma|_M$.\\

Next, consider any $y,z\in B$. Then both $|y-z|$ and $\bm{d}(y,z)$ are less than $\frac{2}{3}i_M$, but any curve $\gamma$ from $y$ to $z$ that leaves $B(x,\frac{7}{9}i_M)$ has both Euclidean length and $M$-length at least $\frac{8}{9}i_M$. It follows that 
\[
|y-z| =\inf\{|\gamma|_E\} \quad\text{and}\quad \bm{d}(y,z)= \inf\{|\gamma|_M\}
\]
where both infimums are taken over the same set of curves, namely the set of piecewise smooth curves $\gamma$ that go from $y$ to $z$ and are contained in $B(x,\frac{7}{9}i_M)$. The point is that for all such $\gamma$, $|\gamma|_E \asymp_{\frac{7}{9}i_M} |\gamma|_M$, and thus it follows that $\bm{d}(y,z)$ and $|y-z|$ are comparable as claimed in the lemma.
\end{proof}

We will also need an estimate on covariant derivatives of the heat kernel. First, it is well-known (see \cite{LiYau86,WithLudoAIHP}) that for each $n\geq 1$, there is a constant $C_n>0$, depending only on $n$, such that
\[
|\nabla^nP_t(x,y)|\leq C_n\left[\frac{\bm{d}(x,y)}{t}+\frac{1}{\sqrt{t}}\right]^nP_t(x,y)
\]
for all $x,y\in M$ and $t\in(0,1]$. Next, we claim that there exists $C'>0$, depending only on $T$ and $\eps$ (and $M$) such that 
\begin{equation}\label{Eqn:MPBound}
P_t(x,y)\leq C' G^\epsilon_t(\bm{d}(x,y))
\end{equation}
for all $x,y\in M$ and all $t\in(0,T]$. The argument proceeds by two cases. First, suppose that $\bm{d}(x,y)\geq \frac{1}{2}i_M$. Then it is immediate from Lemma \ref{lem: heat kernel bd} and \eqref{est: LY beat Hsu larger eps}
that, for some $C'$, $P_t(x,y)\leq C'G^\epsilon_t(\bm{d}(x,y))$ for all such $x$ and $y$ and all $t\in(0,1)$. But by the smoothness of $P_t$ the compactness of $M$, and the fact that, for any $\delta>0$, $G^\epsilon_t(\bm{d}(x,y))$ is bounded from below on $t\in[\delta,T]$ uniformly in $x$ and $y$, we can extend to all $t\in(0,T]$ by adjusting $C'$. Second, suppose that $\bm{d}(x,y)\leq \frac{1}{2}i_M$. The set of pairs $(x,y)$ that satisfy this inequality form a compact subset of $M\times M$ which is disjoint from the cut locus. Then the Minakshushundaram-Pleijel expansion (see \cite{ML} for a modern derivation) of the heat kernel gives the small-time asymptotic expansion
\[
P_t(x,y) = \frac{1}{t^{d/2}}e^{-\frac{\bm{d}^2(x,y)}{2t}}\lp c_0(x,y)+O(t)\rp
\]
where $c_0$ is a smooth, positive function and the $O(t)$ term is uniform over all such pairs $(x,y)$. It follows that there are $\delta>0$ and $C'>0$ such that $P_t(x,y)\leq C'G^\epsilon_t(\bm{d}(x,y))$ for all $t\in(0,\delta]$ and all $x$ and $y$ with $\bm{d}(x,y)\leq \frac{1}{2}i_M$. But just as before, we can extend to all $t\in(0,T]$ by adjusting $C'$, and combining these two cases establishes \eqref{Eqn:MPBound}. It follows that, after adjusting $C_n$ to depend also on $\eps$ and $T$, we have
\[
|\nabla^n P_t(x,y)|\leq C_n\left[\frac{\bm{d}(x,y)}{t}+\frac{1}{\sqrt{t}}\right]^n G^\epsilon_t(\bm{d}(x,y))
\]
for all $x,y\in M$ and $t\in(0,T]$. (Again, we use that $|\nabla_y^nP_t(x,y)|$ is bounded from above for $t\in[1,T]$ while $ G^\epsilon_t(\bm{d}(x,y))$ is bounded from below to extend this bound to all $t\in(0,T]$.)\\

Next, we specialize to the case when $n=1$ and use the fact that, for any positive $\delta$, $\frac{d}{t}e^{-\frac{d^2}{\delta t}}$ is bounded over $(d,t)\in [0,D]\times (0,\infty)$ in order to absorb the $\frac{\bm{d}(x,y)}{t}P_t(x,y)$ term into the Gaussian $G^\epsilon_t$, at the cost of modifying the constant $C_1$ again. Putting this together, we have that
\begin{equation}\label{Eqn:CovDerBouund}
|\nabla P_t(x,y)|\leq \frac{C}{\sqrt{t}}G^\epsilon_t(\bm{d}(x,y))
\end{equation}
for all $x,y\in M$ and $t\in(0,T]$, where $C$ depends on $\eps$ and $T$ (as well as $M$, of course).\\

With the preliminaries out of the way, we come to our basic estimate on the spatial increments of the heat kernel.\\

\begin{lemma}\label{Lem:HeatIncrement}
Fix an $\epsilon>0$. For all $\theta\in(0,1]$ and $T>0$, there exist $C>0$ and $\bm{d}_0>0$ such that, for all $t\in(0,T]$, $x,y,z\in M$, and $\theta\in(0,1]$, we have
\begin{equation}\label{est:HK space inc}
    |P_t(x,y)-P_t(x,z)|\leq C\frac{\bm{d}(y,z)^\theta}{t^{\theta/2}}\lb G^\epsilon_t(\bm{d}(x,y))+G^\epsilon_t(\bm{d}(x,z))+G^\epsilon_t(\bm{d}_0)\rb.
\end{equation}
\end{lemma}

\begin{proof}
Our first task is to show that, for an appropriate $\bm{d}_0$, there is a $c>0$ such that, for any $x,y,z\in M$, there is a (piecewise smooth) path $\gamma$ from $y$ to $z$ such that the length of $\gamma$ is no more than $c\bm{d}(y,z)$ and such that the distance from $\gamma$ to $x$ is never less than $\min\{ \bm{d}(x,y), \bm{d}(x,z), \bm{d}_0 \}$. We let $\bm{d}_0$ be $\frac{1}{6}i_M$, and we choose some $x\in M$. For convenience, we let $m$ be $\min\{ \bm{d}(x,y), \bm{d}(x,z), \bm{d}_0 \}$.\\

We consider three cases. First suppose that both $\bm{d}(x,y)$ and $\bm{d}(x,z)$ are less than or equal to $\frac{1}{3}i_M$. 
Consider the line segment $\ell$ from $y$ to $z$, which is the unique shortest path with respect to the standard Euclidean metric on $B)$ and which thus realizes $|y-z|$. Note that $\ell$ lies entirely in $B$, and if any part of $\ell$ is at distance less than $m$ from $x$, it will be a sub-segment of $\ell$. Then we can replace this subsegment by an arc of a circle of radius $m$ centered at $x$ (this is just replacing a chord by the arc it subtends, in polar coordinates). If the angle of this arc is $\phi\in[0,\pi]$, then the ratio of the arc length to the chord length is $\frac{\phi}{2\sin(\phi/2)}$, which is always between 1 and $\frac{\pi}{2}$. Now, we let $\gamma$ be $\ell$ with this modification of replacing the chord by the associated arc. Then $\gamma$ is piecewise smooth, and the distance from $\gamma$ to $x$ is never less than $m$. Moreover, we just saw that $|\gamma|_E$ is comparable with $|\ell|_E=|y-z|$ (with universal implicit constants), and combing this with Lemma \ref{Lem:MetricComp}, we see that $|\gamma|_M$ is comparable to $\bm{d}(y,z)$ (with implicit constants depending only on $M$ and the explicit choices we've made). This gives $\gamma$ as desired in this case.\\

The second case is when $\bm{d}(y,z)\leq \frac{1}{6}i_M$ and the the first case doesn't apply, which means that, without loss of generality, we can assume that $\bm{d}(x,y) > \frac{1}{3}i_M$. Then the triangle inequality implies that $\bm{d}(x,z)> \frac{1}{6}i_M$, and it follows that $m=\bm{d}_0$. Further, we know there is a unique minimal geodesic from $y$ to $z$; we take this as $\gamma$. Then the triangle inequality also implies that $\gamma$ is always at least distance $\frac{1}{6}i_M$ from $x$, and we see that $\gamma$ has the desired properties.\\

For the final case, we suppose that $\bm{d}(y,z) > \frac{1}{6}i_M$. Then we claim that there is some $C>0$ such that we can find a curve $\gamma$ from $y$ to $z$ that never gets closer to $x$ than $m$ an such that $|\gamma|_M\leq C$. Such a $\gamma$ would then have the required properties. An elementary, if not necessarily efficient, way to construct such a $\gamma$ is as follows. Starting from $y$, travel along a radial geodesic from $x$ to the sphere of radius $m$ around $x$. The distance travelled is bounded by the diameter of $M$. Because $m\leq \bm{d}_0$, the sphere of radius $m$ around $x$ is inside $B(x,\frac{1}{3}i_M)$, so as in the first case, we can travel an arc of a circle of radius $m$ until we're on a radial geodesic that takes us to $z$. The distance traveled on the circle of radius $m$ is bounded as in the first case, and the last segment has length bounded by the diameter of $M$, so this gives a piecewise smooth $\gamma$ as desired.\\

With this bound in hand, choose any $x,y,z$, and $T$ as in the first part of the lemma. Then integrating \eqref{Eqn:CovDerBouund} along $\gamma$, and using that $G^\epsilon_t(r)$ is decreasing in $r$, gives
\begin{equation}\label{est:HK space inc2}
    |P_t(x,y)-P_t(x,z)|\leq C\frac{\bm{d}(y,z)}{\sqrt{t}}[G^\epsilon_t(\bm{d}(x,y))+G^\epsilon_t(\bm{d}(x,z))+G^\epsilon_t(\bm{d}_0)].
\end{equation}
Here we've absorbed the constant $c$ into $C$, in keeping with our convention, and we note that $C$ depends only on $M$, $\eps$ and $T$. For any $\theta\in(0,1]$, we can apply \eqref{est:HK space inc2} and obtain
\begin{equation}\begin{split}\label{Eqn:ThetaBound}
    |P_t(x,y)-P_t(x,z)|&=|P_t(x,y)-P_t(x,z)|^\theta|P_t(x,y)-P_t(x,z)|^{1-\theta}\\
    &\leq \left(C\frac{\bm{d}(y,z)}{\sqrt{t}}[G^\epsilon_t(\bm{d}(x,y))+G^\epsilon_t(\bm{d}(x,z))]\right)^\theta \\
    & \qquad\times |P_t(x,y)-P_t(x,z)|^{1-\theta}.
\end{split}\end{equation}\\

Since \eqref{Eqn:MPBound} applies to both $P_t(x,y)$ and $P_t(x,z)$, we have
\[
|P_t(x,y)-P_t(x,z)| < 2C' |G^\epsilon_t(\bm{d}(x,y))+G^\epsilon_t(\bm{d}(x,z))+ G^\epsilon_t(\bm{d}_0) | .
\]
Then using this in \eqref{Eqn:ThetaBound}, we find
\[
 |P_t(x,y)-P_t(x,z)|\leq C^{\theta}(2C')^{1-\theta} \frac{\bm{d}(y,z)^\theta}{t^{\theta/2}}\lb G^\epsilon_t(\bm{d}(x,y))+G^\epsilon_t(\bm{d}(x,z))+G^\epsilon_t(\bm{d}_0)\rb.
\]
But since $C^{\theta}(2C')^{1-\theta}$ is bounded for $\theta\in [0,1]$ and both $C$ and $C'$ depend only on $\eps$ and $T$ (and $M$) , this implies \eqref{est:HK space inc}.
\end{proof}

We can now prove our main result about the continuity of $u(t,x)$.\\

\begin{lemma}[H\"older Continuity of Solutions]\label{lem: Holder cont}

Let $\alpha>{(d-2)}/{2}$ (equivalently, $\nu >0$ where $\nu=2\alpha+2-d$). Fix $\varepsilon\in(0,1)$ and $\varepsilon<T$. Then for every $p\ge 2$ there is a constant $C$ depending on $p,\varepsilon$ and $T$ such that for all
$s,t\in[\varepsilon,T]$ and $x,y\in M$,
\begin{equation}\label{eq:increment-moment}
\mathbb{E}\big|u(t,x)-u(s,y)\big|^p
\le C\Big(|t-s|^{\frac{p\nu}{4}} + \bm d(x,y)^{\frac{p\nu}{2}}\Big).
\end{equation}
Consequently, $u$ admits a modification that is jointly continuous on
$[\varepsilon,T]\times M$, and for any exponents
\[
\gamma_t\in\bigl(0,\min(\nu/4,1)\bigr),\qquad \gamma_x\in\bigl(0,\min(\nu/2,1)\bigr),
\]
the modification is H\"older of order $\gamma_t$ in time and $\gamma_x$ in space
(on $[\varepsilon,T]\times M$).
\end{lemma}

\begin{proof}
Fix an $\epsilon>0$. It is clear that we only bound the moments of increments of $I(t,x)$ defined in \eqref{pam:mild}, from which applying Burkholder-Gundy-Davies and Minkowski's integral inequality gives us for all \(x,y\in M,\varepsilon\le s<t\le T\)\begin{align*}
    \E[|I(t,x)-I(s,y)|^p]^{\frac{2}{p}}\leq& C(I(t;x,y)+I(0,s;y)+I(s,t;y)),
\end{align*}
where
\begin{align*}
    I(t;x,y):=&\int_0^t\iint_{M^2}|G_{\alpha,\rho}(z,z')|\E[|u(\tau,z)|^p]^{\frac{1}{p}}\E[|u(\tau,z')|^p]^{\frac{1}{p}}\\
    &\times [P_{t-\tau}(x,z)-P_{t-\tau}(y,z)][P_{t-\tau}(x,z')-P_{t-\tau}(y,z')]m(dz)m(dz')d\tau,\\
    I(0,s;y):=&\int_0^s\iint_{M^2}|G_{\alpha,\rho}(z,z')|\E[|u(\tau,z)|^p]^{\frac{1}{p}}\E[|u(\tau,z')|^p]^{\frac{1}{p}}\\
    &\times [P_{t-\tau}(y,z)-P_{s-\tau}(y,z)][P_{t-\tau}(y,z')-P_{s-\tau}(y,z')]m(dz)m(dz')d\tau,\\
    I(s,t;y):=&\int_s^t\iint_{M^2}|G_{\alpha,\rho}(z,z')|\E[|u(\tau,z)|^p]^{\frac{1}{p}}\E[|u(\tau,z')|^p]^{\frac{1}{p}}\\
    &\times P_{t-\tau}(y,z)P_{t-\tau}(y,z')m(dz)m(dz')d\tau,
\end{align*}
The following three estimates will establish the increment moment estimates of desired result, and the rest of the proof will be dedicated to proving them. The H\"older regularity results follow from applying Kolmogorov-Chentsov.\begin{align}
    I(s,t;y)\le& C_{\varepsilon,T,p}|t-s|^{1+\alpha-\frac{d}{2}}, \label{Eqn:Isty}\\
    I(t;,x,y)\le& C_{\varepsilon,T,p,\alpha} \bm{d}(x,y)^{2\theta},\qquad 0<\theta<\frac{2\alpha-d}{2}+1, \label{Eqn:Itxy}\\
    I(0,s;y)\le& C_{\varepsilon,T,p}|t-s|^{1+\alpha-\frac{d}{2}} \label{Eqn:I0sy}.
\end{align}
First, apply \eqref{est:HK space inc} in $I(t;x,y)$ and Proposition \ref{Prop:LiYau} in $I(s,t;y)$ to the appropriate heat kernel differences or heat kernels. 
One then applies the moment estimates of $u$ from Theorem \ref{thm: main result} and Proposition \ref{Prop: G_alpha} to the covariance.\\

For \eqref{Eqn:Isty}, the fact that $\tau\ge s\ge \varepsilon>0$ and $\mu(M)<\infty$ gives us \(\sup_{\tau\ge s,z\in M}\E[|u(\tau,z)|^p]\leq C_{\varepsilon,p}\). This combined with \eqref{est:2int 2LY Riesz} gives us \[
    I(s,t;y)\leq C_{\varepsilon,p}\int_s^t (t-\tau)^{\frac{2\alpha-d}{2}}+1d\tau=C_{\varepsilon,p}(|t-s|^{\frac{2\alpha-d}{2}+1}+|t-s|)\le C_{\varepsilon,p,T}|t-s|^{1+\alpha-\frac{d}{2}}.
\]

We now move to show \eqref{Eqn:Itxy}. Applying \eqref{est:HK space inc} to $I(t;x,y)$ gives us\begin{align*}
    I(t;x,y)\le& C\int_0^t\iint_{M^2}|G_{\alpha,\rho}(z,z')||\E[|u(\tau,z)|^p]^{\frac{1}{p}}\E[|u(\tau,z')|^p]^{\frac{1}{p}}\frac{\bm{d}(x,y)^{2\theta}}{(t-\tau)^{\theta}}\\
    &\times [G^\epsilon_{t-\tau}(\bm{d}(x,z))+G^\epsilon_{t-\tau}(\bm{d}(y,z)){+G^\epsilon_{t-\tau}(d_0)}]\\
    &\times[G^\epsilon_{t-\tau}(\bm{d}(x,z'))+G^\epsilon_{t-\tau}(\bm{d}(y,z')){+G^\epsilon_{t-\tau}(d_0)}]m(dz)m(dz')d\tau
\end{align*}
It remains to show that the right hand side above is bounded for appropriate values of $\theta$. As before, we apply Theorem \ref{thm: main result} to the moments in the integrand and Proposition \ref{Prop: G_alpha} to the covariance.
We then have \[I(t;x,y)\le I_<(t;x,y)+I_>(t;x,y),\] where $I_<(t;x,y)$ has the same integrand as the right hand side above but with the temporal integration domain is $(0,\frac{t}{2}]$, and likewise for $I_>(t;x,y)$, except its temporal integration domain is $(\frac{t}{2},t)$.\\
If $\tau\in(0,\frac{t}{2})$, we have \(t-\tau\ge \frac{t}{2}\ge \frac{\varepsilon}{2}\), giving us \[(t-\tau)^{\theta}\le (\varepsilon/2)^\theta =C_{\varepsilon,\theta},\text{ and for all }r\ge0,\quad G^\epsilon_{t-\tau}(r)\le C(t-\tau)^{-\frac{d}{2}}=C(\varepsilon/2)^{-\frac{d}{2}}=C_\varepsilon. \]
An application of \eqref{est:2int 2LY Riesz} with the above then gives us for all $\theta\in(0,1)$ \[I_<(t;x,y)\le C_\varepsilon \bm{d}(x,y)^{\theta}.\]
For $I_>(t;x,y)$, since $\tau\in (\frac{t}{2},t)$, by Theorem \ref{thm: main result} we have \[E[|u(\tau,x)|^p]^{\frac{1}{p}}\le C_pJ_\mu(\tau,x)e^{\theta(\beta)\tau}\le C_p\mu(M)(\varepsilon/2)^{-\frac{d}{2}} e^{\theta(\beta)T}=C_{\varepsilon,T,p}.\]
We now write \(I_<(t;,x,y)=I_0(t;x,y)+I_1(t;x,y)\), where $I_0(t;x,y)$ is the sum all the integrals that have a factor of $G^\epsilon_{t-\tau}(d_0)$ in the integrand and $I_1(t;x,y)$ is the only integral which does not have $G^\epsilon_{t-\tau}(d_0)$ in the integrand after doing the appropriate trinomial multiplication.
Since $G^\epsilon_{t-\tau}(d_0)\le C$, applying \eqref{est: Riesz} and \eqref{est:2int 1LY Riesz} along with the integrability of \((t-\tau)^{-\theta}\) over $(\frac{t}{2},t)$ for any $\theta\in (0,1)$ gives \[I_0(t;x,y)\le C_{\varepsilon,T,p}\bm d(x,y)^{2\theta}.\]
For $I_1(t;,x,y)$, application of \eqref{est:2int 2LY Riesz} gives \[I_1(t;x,y)\le C_{\varepsilon,T}\bm{d}(x,y)^{2\theta}\int_{\frac{t}{2}}^t \frac{(t-\tau)^{\frac{2\alpha-d}{2}}+1}{(t-\tau)^{\theta}} d\tau=C_{\varepsilon,T}\bm{d}(x,y)^{2\theta}\int_0^{\frac{t}{2}}\tilde{\tau}^{\frac{2\alpha-d}{2}-\theta}+\tilde{\tau}^{-\theta}  d\tilde{\tau}.\]
Observe that RHS is bounded only when \(\frac{2\alpha-d}{2}-\theta>-1\iff \theta<\frac{2\alpha-d}{2}+1\), giving us for this range of $\theta$ \[I_1(t;x,y)\le C_{\varepsilon,T,\alpha}d(x,y)^{2\theta}.\]
Combining the estimates of $I_<(t;x,y),I_0(t;x,y)$ and $I_1(t;x,y)$ gives us for \(\theta<\frac{2\alpha-d}{2}+1\) \[I(t;x,y)\le I_<(t;x,y)+I_>(t;x,y)=I_<(t;x,y)+I_0(t;x,y)+I_1(t;x,y)\le C_{\varepsilon,T,p,\alpha}\bm{d}(x,y),\] which is exactly \eqref{Eqn:Itxy}.\\

This leaves only \eqref{Eqn:I0sy} unproven. For this, we again apply Theorem \ref{thm: main result} and Proposition \ref{Prop: G_alpha} to $G$, after which we apply the decomposition $\int_0^s=\int_0^{s/2}+\int_{s/2}^s$.

For $\tau\in[0,s/2]$, we have \begin{align*}|P_{t-\tau}(y,z)-P_{s-\tau}(y,z)|&=\left|\int_{s-\tau}^{t-\tau} \partial_rP_r(y,z) dr\right|\leq \int_{s-\tau}^{t-\tau} |\partial_rP_r(y,z)| dr\\
=& \int_{s-\tau}^{t-\tau} |\Delta_zP_r(y,z)| dr\leq C\int_{s-\tau}^{t-\tau} r^{-1}G^{\epsilon}_r(\bm{d}(x,y)) dr\\
\leq& C_\varepsilon|t-s|G^{\epsilon}_{t-\tau}(\bm{d}(y,z)).
\end{align*}
In the above, we used \eqref{Eqn:CovDerBouund} with $n=2$ and the fact that $s\geq \varepsilon$.
Application of \eqref{est:2int 2LY Riesz} and the fact that $\mu$ is a finite measure on $M$ thus shows that $\int_0^{s/2}\leq C_{\varepsilon,M}|t-s|$. 

On the other hand, for $\tau\in [s/2,s]$, that the moment factors are bounded by $\sup_{x,y\in M}P_\tau(x,y)\leq C_M\varepsilon^{-\frac{d}{2}}$ and the fact that $\mu$ is a finite measure on $M$ leaves us to estimate \begin{equation}
    I\left(\frac{s}{2},s;y\right):=\int_{\frac{s}{2}}^s\iint_{M^2}\bm{d}(z,z')^{2\alpha-d} [P_{t-\tau}(y,z)-P_{s-\tau}(y,z)][P_{t-\tau}(y,z')-P_{s-\tau}(y,z')]m(dz)m(dz') d\tau
\end{equation}

To estimate $I(s/2,s;y)$, let
\(
B:=B\!\left(y,\frac{i_M}{3}\right).
\) We decompose
\(
M=B\sqcup B^c,
\)
and hence
\[
M^2
=(B\times B)
\sqcup(B\times B^c)
\sqcup(B^c\times B)
\sqcup(B^c\times B^c).
\]
Accordingly,
\[
I\!\left(\frac{s}{2},s;y\right)
=
I_{BB}
+I_{BB^c}
+I_{B^cB}
+I_{B^cB^c}.
\]

For $z\in B^c$, we have $\bm{d}(y,z)\ge\delta=i_M/3$, giving us \begin{align*}
    &\sup_{z'\in M}\int_{B^c} \bm{d}(z,z')^{2\alpha-d}|P_{t-\tau}(y,z)-P_{s-\tau}(y,z)| m(dz)\\
    \leq& \sup_{z'\in M}\int_{B^c}\bm{d}(z,z')^{2\alpha-d} \int_{s-\tau}^{t-\tau}|\partial_rP_r(y,z)|dr m(dz)\\
    \leq& \sup_{z'\in M}\int_{B^c}\bm{d}(z,z')^{2\alpha-d} \int_{s-\tau}^{t-\tau}r^{-1}G^\epsilon_r(\bm{d}(y,z))dr m(dz)\\
    \leq& \sup_{z'\in M}\int_{B^c}\bm{d}(z,z')^{2\alpha-d} \int_{s-\tau}^{t-\tau}r^{-1}G^{\epsilon}_r(\delta^2/9)dr m(dz)\leq C_{\delta,M}|t-s|.\\
\end{align*}

In the above, we used \eqref{Eqn:CovDerBouund} with $n=2$ for the second inequality and \eqref{est: Riesz} along with $\sup_{r> 0}r^{-1}G^{\epsilon}_r(\delta^2/9)\leq C_\delta$ for the final inequality. 
This along with \[|P_{t-\tau}(y,z)-P_{s-\tau}(y,z)|\leq P_{t-\tau}(y,z)+P_{s-\tau}(y,z),\qquad \int_{M} P_{t-\tau}(y,z)+P_{s-\tau}(y,z) m(dz)=2\] being applied in the $B$ integrals in $I_{BB^c}$ and $I_{B^cB}$ gives us \[I_{BB^c}+I_{B^cB}+I_{B^cB^c}\leq C|t-s|+C|t-s|^2.\]

This leaves us with estimating \(I_{BB}.\) We first apply \[J_\mu(\tau,x)\leq C\mu(M)G^{\epsilon}_\tau(0)\leq C\mu(M)G^{\epsilon}_\frac{s}{2}(0)\le C\mu(M)G^{\epsilon}_{\frac{\varepsilon}{2}}(0)= C_{\varepsilon,M}\]
to remove the pointwise moment factor, leaving us with \[I_{BB}\leq C\int_{s/2}^s \iint_{B^2} \bm{d}(z,z')^{2\alpha-d} |P_{t-\tau}(y,z)-P_{s-\tau}(y,z)||P_{t-\tau}(y,z')-P_{s-\tau}(y,z')| m(dz)m(dz') d\tau.\]

We now decompose \([s/2,s]=[s/2,(2s-t)\wedge s)\sqcup [(2s-t)\wedge s,s]\), and correspondingly, \[I_{BB}=\int_{s/2}^s=\int_{s/2}^{(2s-t)\wedge s}+\int_{(2s-t)\wedge s}^s=:I_{BB}^<+I_{BB}^>.\]
For $I_{BB}^<$, we apply \eqref{Eqn:CovDerBouund} with $n=2$, we have \[|P_{t-\tau}(y,z)-P_{s-\tau}(y,z)|\leq \int_{s-\tau}^{t-\tau}|\partial_rP_r(y,z)|dr\leq  \int_{s-\tau}^{t-\tau}r^{-1}G^{\epsilon}_r(\bm{d}(y,z))dr,\]
which gives us \[I_{BB}^<\leq  C\int_{s/2}^{(2s-t)\wedge s}\iint_{[s-\tau,t-\tau]^2}\frac{1}{rr'}\iint_{B^2}\bm{d}(z,z')^{2\alpha-d}G^{\epsilon}_r(\bm{d}(y,z))G^{\epsilon}_r(\bm{d}(y,z'))m(dz)m(dz')drdr'd\tau.\]
By the Lemma \ref{Lem:MetricComp} (and adopting the same notation as there), we have for all $z,z'\in B$ $\bm{d}(z,z')\asymp_\delta |z-z'|$. Since the volume form is also bounded, we have \begin{align*}
    &\iint_{B^2}\bm{d}(z,z')^{2\alpha-d}G^{\epsilon}_r(\bm{d}(y,z))G^{\epsilon}_{r'}(\bm{d}(y,z'))m(dz)m(dz')\\
    \leq& C_M\iint_{B(0,\delta)^2}|z-z'|^{2\alpha-d}G^\epsilon_r(|z|) G^\epsilon_{r'}(|z'|)dzdz'\\
    \leq& C_M\iint_{\R^{2d}}|z-z'|^{2\alpha-d}G^\epsilon_r(|z|) G^\epsilon_{r'}(|z'|)dzdz'\\
    =& C_M\int_{\R^d} |\xi|^{2\alpha-d}\exp(-c_\epsilon (r+r')|\xi|^2) d\xi\leq C_M(r+r')^{\frac{d}{2}-\alpha}.
\end{align*}

Consequently,
\[
I_{BB}^{<}
\le
C
\int_{s/2}^{(2s-t)\wedge s}
\int_{[s-\tau,t-\tau]^2}
\frac{(r+r')^{\alpha-\frac d2}}{rr'}
\,dr\,dr'\,d\tau.
\]
On the $\tau-$interval,
\(
s-\tau\ge t-s.
\)
Hence, for every
\[
r,r'\in[s-\tau,t-\tau]
=
[s-\tau,s-\tau+t-s],
\]
we have
\[
s-\tau\le r,r'\le 2(s-\tau).
\]
It follows that
\[
\frac{(r+r')^{\alpha-\frac d2}}{rr'}
\le
C(s-\tau)^{-2+\alpha-\frac d2}.
\]
Since the two inner integration intervals both have length \(t-s\),
\[
\begin{aligned}
I_{BB}^{<}
&\le
C(t-s)^2
\int_{s/2}^{(2s-t)\wedge s}
(s-\tau)^{-2+\alpha-\frac d2}\,d\tau.
\end{aligned}
\]
If \(2s-t\le s/2\), the above integration set is empty. Otherwise,
\[
\begin{aligned}
I_{BB}^{<}
&\le
C(t-s)^2
\int_{s/2}^{2s-t}
(s-\tau)^{-2+\alpha-\frac d2}\,d\tau\\
&=
C(t-s)^2
\int_{t-s}^{s/2}
v^{-2+\alpha-\frac d2}\,dv\\
&\le
C(t-s)^2
(t-s)^{-1+\alpha-\frac d2}\\
&=
C(t-s)^{1+\alpha-\frac d2}.
\end{aligned}
\]

We next estimate \(I_{BB}^{>}\). On this region,
\[
s-\tau\le t-s.
\]
Instead of applying \eqref{Eqn:CovDerBouund}, we use the crude bound
\[
\left|P_{t-\tau}(y,z)-P_{s-\tau}(y,z)\right|
\le
P_{t-\tau}(y,z)+P_{s-\tau}(y,z).
\]
Therefore,
\[
\begin{aligned}
I_{BB}^{>}
\le C
\int_{(2s-t)\vee s/2}^{s}
\iint_{B^2}
&d(z,z')^{2\alpha-d}
\bigl(P_{t-\tau}(y,z)+P_{s-\tau}(y,z)\bigr)\\
&\times
\bigl(P_{t-\tau}(y,z')+P_{s-\tau}(y,z')\bigr)
\,m(dz)m(dz')\,d\tau.
\end{aligned}
\]
Applying the Gaussian upper bound and passing to normal coordinates as
above, we obtain a sum of four terms of the form
\[
\iint_{\mathbb R^{2d}}
|z-z'|^{2\alpha-d}
G_{r}^{\epsilon}(|z|)
G_{r'}^{\epsilon}(|z'|)
\,dz\,dz',
\]
where
\[
r,r'\in\{s-\tau,t-\tau\}.
\]
The same Euclidean Riesz--Gaussian computation gives
\[
\iint_{\mathbb R^{2d}}
|z-z'|^{2\alpha-d}
G_{r}^{\epsilon}(|z|)
G_{r'}^{\epsilon}(|z'|)
\,dz\,dz'
\le
C(r+r')^{\alpha-\frac d2}.
\]
Since \(r,r'\ge s-\tau\) and \(\alpha-d/2<0\), each of the four terms is bounded by
\[
C(s-\tau)^{\alpha-\frac d2}.
\]
It follows that
\[
\begin{aligned}
I_{BB}^{>}
&\le
C
\int_{(2s-t)\vee s/2}^{s}
(s-\tau)^{\alpha-\frac d2}\,d\tau.
\end{aligned}
\]
The interval of integration is contained in\(\{\tau\in[s/2,s]:s-\tau\le t-s\},\) and since the integrand is nonnegative,
\[
\begin{aligned}
I_{BB}^{>}
&\le C\int_{s-(t-s)}^{s}(s-\tau)^{\alpha-\frac d2}\,d\tau\\
&=C\int_0^{t-s}v^{\alpha-\frac d2}\,dv=C(t-s)^{1+\alpha-\frac d2},
\end{aligned}
\]
where the final identity is valid by Dalang's condition.

Combining the $I_{BB}^<$ and $I_{BB}^>$ estimates, we conclude that
\[
I_{BB}
=
I_{BB}^{<}+I_{BB}^{>}
\le
C|t-s|^{1+\alpha-\frac d2}.
\]
Together with the estimates for the other three spatial contributions,
this yields
\[
I\left(\frac{s}{2},s;y\right)
\le
C\left(|t-s|+|t-s|^2
+|t-s|^{1+\alpha-\frac d2}\right).
\]
Since
\[
0<1+\alpha-\frac d2<1
\]
and \(0<t-s\le T\), we finally obtain
\[
I\left(\frac{s}{2},s;y\right)
\le
C_T|t-s|^{1+\alpha-\frac d2}.
\]
This gives us \eqref{Eqn:I0sy}, which concludes the proof.
\end{proof}

We will also need the following elementary lemmas. The first shows that the solution can be approximated by solutions with smooth initial data. Recall that for any finite measure $\mu$ on $M$, $J_\mu(t,x)$ denotes the solution to the homogeneous equation with initial data $\mu$, as introduced in \eqref{pam:mild}.

\begin{lemma}\label{lem:initial condition approx}
    Assume $\alpha>{(d-2)}/{2}$. Let $\mu$ be a finite measure on $M$. For $\delta>0$, denote by $u_\delta$ the random field solution to equation \eqref{eq: SHE}  starting from $J_\mu(\delta,x)m(dx)$, and $u$ being the solution starting from $\mu$. We have $$\lim_{\delta\downarrow0}\E[|u(t,x)-u_\delta(t,x)|^2]=0.$$
\end{lemma}

\begin{proof}[Proof of Lemma \ref{lem:initial condition approx}]
    Let $v_\delta:=u-u_\delta$. Since \eqref{eq: SHE} is linear in the initial condition, $v_\delta$ is the solution to equation \eqref{eq: SHE}  starting from $\mu_\delta(dx):=\mu(dx)-J_\mu(\delta,x)m(dx).$ By Theorem \ref{thm: main result}, the function $g_\delta(t,x,x'):=\E[v_\delta(t,x)v_\delta(t,x')]$ satisfies $$g_\delta(t,x,x')=\Tilde{J}_{\mu_\delta}(t,x,x')+\beta^2\iint_{M^2}\mu_\delta(dz)\mu_\delta(dz')K_\beta(t,z,x,z',x').$$
   In addition, Corollary \ref{cor:K beta bound} together with Lemma \ref{lemma:behavior of H} implies that for any fixed $\epsilon>0$ and  all $x,y,x',y'\in M$  we have
   \[K_\beta(t,x,y,x',y')\le C[G^\epsilon_t(\bm{d}(x,y))+1][G^\epsilon_t(\bm{d}(x',y'))+1]e^{ct}.\]
    Here, $C, c>0$ are constants depending on $\beta$ and $M$. The result then follows from the fact that $\mu_{\delta}\to0$ weakly as $\delta\downarrow0$ and the compactness of $M$. 
\end{proof}

The next lemma shows that the solution is weakly continuous in time at $t=0$.
\begin{lemma}\label{lem: cont dep on initial data}
    Suppose $u(t,x)$ is the solution to equation \eqref{eq: SHE} with initial data $\mu$. We then have for all $\phi\in C(M)$, $$\int_M u(t,x)\phi(x)m(dx) \lra \int_M \phi(x)\mu(dx)\text{ in }L^2(\P),\quad\mathrm{as}\ \ t\downarrow0. $$
    
\end{lemma}
\begin{proof}
    Suppose $0<t\leq T$. Since $u(t,x)$ admits the mild form \eqref{pam:mild}, it suffices to prove that
    \[\label{eq: stoch int to 0 small time} \lim_{t\downarrow0}L(t)=0, \quad\mathrm{in}\ \ L^2(\mathbb{P}),
    \]
    where we have set $L(t)=\int_M I(t,x)\phi(x)m(dx)$.
   
    By stochastic Fubini, we have $$L(t)=\int_0^t \int_M \left(\int_M P_{t-s}(x,y)\phi(x) m(dx)\right)u(s,y) W(ds,dy).$$
    It\^o isometry and Theorem \ref{thm: main result} then gives \begin{align*}
        \E[L(t)^2]=&\int_0^t \int_M\int_M\left(\int_M P_{t-s}(x,y)\phi(x) m(dx)\right)\left(\int_M P_{t-s}(x,y')\phi(x) m(dx)\right)\\
        &\times\E[u(s,y),u(s,y')]\bm{G}_{\alpha,\rho}(y,y')m(dy)m(dy')ds\\
        \leq& C_T\norm{\phi}_{L^\infty(M)}^2\int_0^t \int_M\int_MJ_\mu(s,y)J_\mu(s,y')|\bm{G}_{\alpha,\rho}(y,y')|m(dy)m(dy')ds\\
        =& C_T\int_0^t \iint_{M^2}\iint_{M^2}P_s(y,z)|\bm{G}_{\alpha,\rho}(y,y')|P_s(y',z')m(dy)m(dy') \mu(dz)\mu(dz')ds.
    \end{align*}
    Now applying Corollary \ref{Prop:LiYau} for the heat kernel $P_s$ and Proposition \ref{Prop: G_alpha} for the covariance function $\bm{G}_{\alpha,\rho}$, as well as the integral bounds \eqref{est:2int 2LY Riesz} and \eqref{est: Riesz}, we have \[\sup_{z,z'\in M}\iint_{M^2}P_s(y,z)|\bm{G}_{\alpha,\rho}(y,y')|P_s(y',z')m(dy)m(dy')\leq C(s^{\frac{2\alpha-d}{2}}+1).\]
    Since $\mu$ is a finite measure, integrating the upper bound above in $s$ gives us \[\E[L(t)^2]\leq C\mu(M)^2(t^{\frac{2\alpha-d}{2}+1}+t)\,{\rightarrow}\,0, \quad \mathrm{as}\ \ t\downarrow0.\]
    This completes the proof.
\end{proof}
\subsection{Comparison Principles}

In a recent work \cite[Theorem 1.4]{FSY26}, the authors established weak and strong comparison principles for \eqref{eq: SHE} on general metric measure spaces. The following theorem is a restatement of their result in our setting.

\begin{theorem}\label{th: comparision with nice initial}
Let $f_1,f_2\in C(M)$, and denote by $u_i$ the solutions to equation \eqref{eq: SHE} starting from initial conditions $f_i(x)m(dx)$ for $i=1,2$.

\begin{enumerate}
    \item[\textup{(i)}] If $f_1(x)\leq f_2(x)$ for all $x\in M$, then we have
    $$\mathbb{P}\left[u_1(t,x)\leq u_2(t,x)\ \mathrm{for\ all}\ t\geq0, x\in M\right]=1.$$
    
        \item[\textup{(ii)}] If $f_1 \leq f_2$ on $M$ and $f_1(x) < f_2(x)$ for some $x\in M$, then
        $$\mathbb{P}\left[u_1(t,x)< u_2(t,x)\ \mathrm{for\ all}\ t>0, x\in M\right]=1.$$
\end{enumerate}
\end{theorem}

The above comparison theorem can be extended to equation \eqref{eq: SHE} with measure-valued initial condition. To this end, we first introduce some standard notation for comparing measures.
\begin{itemize}
    \item For any measure $\mu$ on $M$, $\mu \ge 0$ means that $\mu$ is nonnegative, that is, $\mu(A) \ge 0$ for all Borel sets $A$. We write $\mu > 0$ if, in addition, there exists a Borel set $A$ such that $\mu(A) > 0$.
    
    \item For two measures $\mu,\nu$ on $M$, $\mu\ge\nu$ means that $\mu - \nu \ge 0$, and $\mu>\nu$ means that $\mu - \nu > 0$.
\end{itemize}

\begin{theorem}[Weak Comparison Principle]\label{thm:Weak Comparison}
Suppose $\alpha > {(d-2)}/{2}$, and let $\mu_1,\mu_2$ be finite Borel measures on $M$ satisfying $\mu_1 \le \mu_2$. 
Denote by $u_i(t,x)$, $i=1,2$, the solution to equation \eqref{eq: SHE} driven by the same noise but with initial condition $\mu_i$, $i=1,2$, respectively. Then
\[
\P\bigl[u_1(t,x) \le u_2(t,x) \text{ for all } t>0,\ x\in M\bigr] = 1.
\]
\end{theorem}

\begin{proof}
    For $\delta>0$, let $u_{\delta,i},i=1,2$ be the solutions to \eqref{eq: SHE}  starting from $J_{\mu_i}(\delta,x)m(dx)$ and set $$v_\delta(t,x):=u_{\delta,2}(t,x)-u_{\delta,1}(t,x).$$ By Theorem \ref{th: comparision with nice initial}-(i), we have
    $$\P[v_\delta(t,x)\geq 0]=1\text{ for all }t>0,x\in M.$$
   Now the claim follows from Lemma~\ref{lem:initial condition approx} and the continuity of $u(t,x)$ established in Lemma~\ref{lem: Holder cont}.
\end{proof}

\begin{theorem}[Strong Comparison Principle]\label{thm:Strong Comparison}
    Suppose $\alpha>{(d-2)}/{2}$, and let $\mu_1,\mu_2,u_1,u_2$ be the same as in Theorem \ref{thm:Weak Comparison}.
    Then, $\mu_1<\mu_2$ implies \[\P[u_1(t,x)< u_2(t,x)\text{ for all }t>0,x\in M]=1.\]
\end{theorem}
\begin{proof}
    Again let $u(t,x):=u_2(t,x)-u_1(t,x)$. It suffices to show that for each fixed $\varepsilon>0$, \begin{align}\label{eq:positive at positive time}\P[u(t,x)>0, \text{ for all }t\geq \varepsilon,x\in M]=1.\end{align}
   
    We first prove by contradiction that
    \begin{equation}
        \label{eqn: prob positive time not 0} \P[u(\varepsilon,x)=0 \text{ for all }x\in M]=0.
    \end{equation}
    Indeed, since $u(t,x)$ is continuous in $x$ almost surely by Lemma \ref{lem: Holder cont}, weak comparison(Theorem \ref{thm:Weak Comparison}) gives us $u(t,x)\geq 0$ a.s. Thus if \eqref{eqn: prob positive time not 0} is false, by the Markov property and Theorem \ref{th: comparision with nice initial}-(ii), we must have, at all times $\varepsilon'\in [0,\varepsilon]$,  $u(\varepsilon',x)=0$ for all $x\in M$, with some strict positive probability. But this contradicts Lemma \ref{lem: cont dep on initial data} as $\varepsilon'\downarrow0$, and hence completes the proof of \eqref{eqn: prob positive time not 0}.
    \smallskip
    
   Thanks to \eqref{eqn: prob positive time not 0}, there is a subspace $\Omega'\subset\Omega$ with full probability such that for all $\omega\in \Omega'$, there exists $x\in M$ such that $u(\varepsilon,x,\omega)>0$. In addition, $x\mapsto u(\varepsilon,x,\omega)$ is continuous thanks to Lemma \ref{lem: Holder cont}. Next, we set $v(t,x):=u(t+\varepsilon,x)$. By the Markov property of $u$, $v$ satisfies
    \begin{equation}\label{eq: shifted SHE}
        v(t,x)=\int_M P_t(x,y)u(\varepsilon,y)m(dy)+\beta\int_0^t \int_M  P_{t-s}(x,y)v(s,y)W_\varepsilon(ds,dy),
    \end{equation}
    where $W_\epsilon(t,x):=W(t+\epsilon,x)$ is the time-shifted noise. Then Theorem \ref{th: comparision with nice initial}-(ii) implies that $$\P[v_\omega(t,x)>0\text{ for all }t>0,x\in M]=1,$$ where $v_\omega$ is the solution to \eqref{eq: shifted SHE} starting from $u(\epsilon,x,\omega)$. Therefore \eqref{eq:positive at positive time} holds true, which finishes the proof.
\end{proof}

The following corollary follows directly from Theorem~\ref{thm:Strong Comparison} and Lemma~\ref{lem: Holder cont}, and will be used in the proof of the moment lower bound.
\begin{corollary}[Strict Positivity]\label{cor: strict pos}
    Suppose $u(t,x)$ is the solution to \eqref{eq: SHE} with initial condition $\mu>0$. Fix any $t>0$. We have for every $\varepsilon>0$ sufficiently small, $$\P\left[\inf_{x\in M}u(t,x)\geq\varepsilon\right]>0.$$
   
\end{corollary}

\subsection{Exponential lower bound for Moments}

We are now ready to state and prove a matching exponential lower bound for the moments. Recall that $m_0$ is the volume of the manifold $M$ and $\rho\geq0$ is introduced in the definition of the noise \eqref{E:NoiseCov} and \eqref{def: G_alpha etc}.

\begin{theorem}\label{thm: lower bound}
    Assume $\alpha>{(d-2)}/{2}$ and $\rho>0$, and $u(t,x)$ solves equation \eqref{eq: SHE}  starting from $\mu>0$. Then, we have $$\liminf_{t\uparrow+\infty} \frac{1}{t}\ln\E[u(t,x)^2]\geq \frac{\beta^2\rho}{m_0}.$$
\end{theorem}

\begin{remark}
    For $p\geq2$, by the elementary relation $\|u\|_{L^p(\mathbb{P})}\geq\|u\|_{L^2(\mathbb{P})}$, one immediately obtains an exponential lower bound for the $p$-th moment of $u$.
\end{remark}

\begin{proof}[Proof of Theorem \ref{thm: lower bound}]
  
Fix $T>0$ and set $v(t,x):=u(t+T,x)$  as in the proof of Theorem \ref{thm:Strong Comparison}. $v$ is the solution to \eqref{eq: SHE} starting from the random initial condition $u(T,x)$ with time shifted noise $W_T(t,x):=W(t+T,x)$. Clearly it suffices to prove the result for $v$. We first condition on the event $\{\inf_{x\in M}u(T,x)\geq\varepsilon\}$ and employ the elementary lower bound 
    \begin{align}\label{eqn: lower bound by conditional lower bound}\E[v(t,x)^2]\geq \E\left[v(t,x)^2\ \big|\ {\inf_{x\in M}}u(T,x)\geq\varepsilon\right]\ {\times\ \mathbb{P}\left[\inf_{x\in M}u(T,x)\geq\varepsilon\right]}.\end{align}
Note that on the event $\{\inf_{x\in M}u(T,x)\geq\varepsilon\}$, $v(t,x)$ is a solution to \eqref{eq: SHE} with function-valued initial data uniformly bounded from below. One can therefore apply \cite[Theorem 5.1]{CO25} and obtain
 \begin{align}\label{eqn: conditional lower bound}\E\left[v(t,x)^2\ \big|\ {\inf_{x\in M}}u(T,x)>\varepsilon\right]\geq \varepsilon^2 e^{\frac{\beta^2\rho}{m_0}t}.
    \end{align}
    In addition,  we have by Corollary \ref{cor: strict pos},
    \begin{align}\label{eqn: pisitve proba}\mathbb{P}\left[\inf_{x\in M}u(T,x)\geq\varepsilon\right]>0\end{align}
    Plugging \eqref{eqn: conditional lower bound} and \eqref{eqn: pisitve proba} into \eqref{eqn: lower bound by conditional lower bound} finishes the proof.
\end{proof}

\section*{Acknowledgements} 
We thank Fabrice Baudoin for pointing out an error in an earlier version of this work. H.C. is supported by a research grant (VIL73729) from Villum Fonden. {C.O. was partially supported by a Collaboration Grant for Mathematicians
(\#851792) from the Simons Foundation. }
\printbibliography

\end{document}